\documentclass[12pt,reqno]{amsart}
\usepackage{amssymb}
\usepackage{mathtools}
\usepackage{txfonts}
\usepackage{eucal}
\usepackage{extarrows}
\usepackage{shuffle}
\usepackage[all]{xy}
\usepackage{tikz}
\usepackage{graphicx}
\usepackage{float}
\usepackage{xspace}
\usepackage[a4paper,body={16.3cm,22.8cm},centering]{geometry}
\usepackage[colorlinks,final,backref=page,hyperindex,pdftex]{hyperref}
\usepackage[thicklines]{cancel}

\newcommand{\nc}{\newcommand}
\newcommand{\delete}[1]{}

\nc{\mlabel}[1]{\label{#1}}  
\nc{\mcite}[1]{\cite{#1}}  
\nc{\mref}[1]{\ref{#1}}  
\nc{\mbibitem}[1]{\bibitem{#1}} 

\delete{
\nc{\mlabel}[1]{\label{#1}  
{\hfill \hspace{1cm}{\small\tt{{\ }\hfill(#1)}}}}
\nc{\mcite}[1]{\cite{#1}{\small{\tt{{\ }(#1)}}}}  
\nc{\mref}[1]{\ref{#1}{{\tt{{\ }(#1)}}}}  
\nc{\mbibitem}[1]{\bibitem[\bf #1]{#1}} 
}
\newtheorem{theorem}{Theorem}[section]
\newtheorem{prop}[theorem]{Proposition}
\newtheorem{lemma}[theorem]{Lemma}
\newtheorem{coro}[theorem]{Corollary}

\theoremstyle{definition}
\newtheorem{defn}[theorem]{Definition}
\newtheorem{remark}[theorem]{Remark}

\newtheorem{prop-def}{Proposition-Definition}[section]

\newcommand\cal[1]{\mathcal{#1}}

\newcommand\alphlist{a,b,c,d,e,f,g,h,i,j,k,l,m,n,o,p,q,r,s,t,u,v,w,x,y,z}
\newcommand\Alphlist{A,B,C,D,E,F,G,H,I,J,K,L,M,N,O,P,Q,R,S,T,U,V,W,X,Y,Z}
\newcommand\getcmds[3]{\expandafter\newcommand\csname #2#1\endcsname{#3{#1}}}
\makeatletter
\@for\x:=\alphlist\do{\expandafter\getcmds\expandafter{\x}{frak}{\mathfrak}}
\@for\x:=\Alphlist\do{\expandafter\getcmds\expandafter{\x}{frak}{\mathfrak}}
\makeatother

\nc{\bfk}{{\bf k}}
\font\cyr=wncyr10

\newfont{\scyr}{wncyr10 scaled 550}
\nc{\sha}{\mbox{\cyr X}}
\nc{\ssha}{\mbox{\bf \scyr X}}

\nc{\id}{\mathrm{id}}
\nc{\Id}{\mathrm{Id}}
\nc{\lbar}[1]{\overline{#1}}
\nc{\ot}{\otimes}
\nc{\dep}{\mathrm{dep}}

\nc{\tred}[1]{\textcolor{red}{#1}} \nc{\tgreen}[1]{\textcolor{green}{#1}}
\nc{\tblue}[1]{\textcolor{blue}{#1}} \nc{\tpurple}[1]{\textcolor{purple}{#1}}

\nc{\li}[1]{\tpurple{\underline{Li:}#1 }}
\nc{\liadd}[1]{\tpurple{#1}}
\nc{\xing}[1]{\tblue{\underline{Xing:}#1 }}
\nc{\dominique}[1]{\tblue{\underline{Dominique: }#1 }}
\nc{\yuan}[1]{\tred{\underline{Yuan:}#1 }}
\nc{\markus}[1]{\tred{\underline{Markus:} #1}}
\nc\hu[1]{\tpurple{\underline{Huhu:}#1}}
\nc\Hu[1]{\tgreen{#1}}

\usetikzlibrary{calc}
\newlength\xch
\newsavebox\dbox
\sbox\dbox{\tikz{\fill (0,0) circle (0.05cm);}}
\newif\ifqdd
\newif\ifzdd
\newcommand\scopeclip[1]{\begin{scope}
\clip(-1.1,-0.5)rectangle(1.1,1);#1\end{scope}}
\newcommand\XX[2][]{%
\tikz[line width=0.15ex,x=0.5cm,y=0.5cm,baseline,inner sep=1.5pt,
every node/.style={font=\scriptsize},#1]{
\scopeclip{\draw (135:1.5)--(0,0)--(45:1.5) (0,-0.5)--(0,0);}#2}}
\newcommand\xx[3]{%
\scopeclip{\draw(#1/10,#2/10)--+(#3*45:2.5);}}
\newcommand\xxl[2]{\xx{#1}{#2}3}
\newcommand\xxr[2]{\xx{#1}{#2}1}
\newcommand\xxlr[2]{\xxl{#1}{#2}\xxr{#1}{#2}}
\newcommand\xxh[6]{
\draw(#1/10,#2/10)+(0.5*#3*45+0.5*#4*45:#6) node[above] {$#5$};}
\newcommand\xxhu[4][0.15]{\xxh{#2}{#3}13{#4}{#1}}
\newcommand\stree[1]{\XX{\xxhu[0.25]00{#1}}}

\newcommand\Dtree[1]{%
\tikz[line width=0.15ex,x=0.5cm,y=0.5cm,baseline]{%
\draw(0.7,0.7)--(0,0)--(-0.7,0.7)(0,-0.5)--(0,0);
\node[above right,inner sep=0pt] at (-0.25,0.55) {$\scriptstyle#1$};}}

\newcommand\DDDtree[2]{%
\tikz[line width=0.15ex,x=0.5cm,y=0.5cm,baseline]{%
\draw(0,-0.5)--(0,0)(15:1.3)--(0,0)--(165:1.3)(70:1.3)--(0,0)--(110:1.3);
\node[below left,inner sep=0pt] at (110:1.3) {$\scriptstyle#1$};
\node[below right,inner sep=0pt] at (70:1.3) {$\scriptstyle#2$};
\node at (0.05,1) {$\scriptstyle\cdots$};
}}

\nc{\dnx}{\Delta_n A} \nc{\dx}{\Delta A} \nc{\dgp}{{\rm deg_{P}}}
\nc{\dgt}{{\rm deg_{T}}} \nc{\dg}{{\rm deg}} \nc{\ida}{ID($A$)} \nc{\tu}{\tilde{u}} \nc{\tv}{\tilde{v}}
\nc{\nr}{\calr_n} \nc{\nz}{\calz_n} \nc{\fun}{\cala_{n,d}}
 \nc{\fbase}{\calb} \nc{\LF}{\mathrm{RF}} \nc{\FFA}{\mathrm{LF}} \nc{\irr}{\mathrm{Irr}}
 \nc{\result}{\bfk\mathrm{Irr}(S_n)}  \nc{\I}{I_{\mathrm{ID},n}^0}
 \nc{\nrs}{\calr_n^\star} \nc{\ii}{\mathrm{I}} \nc{\iii}{\mathrm{II}}
\nc{\intl}{{\rm int}}\nc{\ws}[1]{{#1}}\nc{\deleted}[1]{\delete{#1}}\nc{\plas}{placements\xspace}

\nc{\bim}[1]{#1}  \nc{\shaop}{\sha_{\Omega}^{+}}  \nc{\shao}{\sha_{\Omega}}
\nc{\bbim}[2]{#1 #2} \nc{\bbbim}[2]{#1,\, #2} \nc{\RBF}{{\rm RBF}}
\nc{\frb}{F_{\RB}} \nc{\shaf}{\ssha_{\tiny{\Omega}}} \nc{\sham}{\diamond_{\tiny{\Omega}}}
\nc{\lf}{\lfloor} \nc{\rf}{\rfloor} \nc{\shan}{\ssha_{\lambda}}
\nc{\rlex}{{\rm {lex}}} \nc{\bb}{\Box} \nc{\ra}{\rightarrow}
\nc{\e}{{\rm {e}}}
\nc{\DDF}{\mathrm{DD}(X,\,\Omega)}\nc{\DTF}{\mathrm{DT}(X,\,\Omega)} \nc{\DT}{\mathrm{DT}'(\Omega,\,V)}
\nc{\bra}{\mathrm{bra}} \nc{\bre}{\mathrm{bre}}
\nc{\dec}{\mathrm{dec}} \nc{\diamondw}{\diamond_{w}}
\nc{\type}{\mathrm{type}}
\nc{\ddx}{\mathrm{DD}(X)}
\nc{\dfx}{\mathrm{DT}(\Delta X)}

\nc\caF[1]{\cal{F}_{#1}(X,\,\Omega)}
\nc\calt{\cal{T}(X,\,\Omega)} \nc\caltn{\cal{T}_n(X,\,\Omega)}
\nc\caltbin{\cal{T}_b(X,\,\Omega)}
\nc\calta{\cal{T}_0(X,\,\Omega)}
\nc\caltb{\cal{T}_1(X,\,\Omega)}
\nc\caltc{\cal{T}_2(X,\,\Omega)}
\nc\caltd{\cal{T}_3(X,\,\Omega)}
\nc\caltm{\cal{T}_m(X,\,\Omega)}
\nc\calf{\cal{F}(X,\,\Omega)}
\nc\fram{\frak{M}(\Omega,\, X)}
\nc\shaw{\sha^{NC}_w(\Omega,\, X)}
\nc\dw{\diamond_w} \nc\dl{\diamond_\ell}
\nc\shal{\sha^{NC}_\ell(X,\, \Omega)} \nc\shav{\sha^{NC}_w(\Omega,\, V)} \nc\shat{\sha^{NC,1}_w(\Omega,\, T^{+}(V))}
\nc{\cfo}{\cal{F}(X,\,\Omega)}

\nc{\lar}{\varinjlim}
\nc\XO{(X,\,\Omega)}
\def\cxo#1#2;{\cal{#1}#2\XO}
\def\cxob#1#2;{\cal{#1}#2_b\XO}
\nc\lrf[2]{B_{#2}^+(#1)}
\nc{\fd}{\mathrm{\text{typed angularly decorated planar rooted trees}}}
\nc{\rb}{\mathrm{RBFWs}} \nc{\dfw}{\mathrm{DFW{(X)}}} \nc{\tfw}{\mathrm{TFW{(X)}}}
\nc{\tfv}{\mathrm{TFW{(V)}}} \nc{\rbf}{\mathrm{RBF}}
\nc{\db}{\mathrm{db}}
\nc{\st}{\mathrm{st}}

\def\ee#1;{\vee_{#1}}
\def\Ve#1,#2,#3;{\vee_{#1,\,(#2,\,#3)}}
\def\bigvv#1;#2;{\bigvee\nolimits_{#1}^{#2}}
\def\bigv#1;#2;#3;{\bigvee\nolimits_{#1}^{#2;\,#3}}

\nc{\Irr}{\mathrm{Irr}} \nc{\lc}{\lfloor} \nc{\rc}{\rfloor}
\nc{\rswx}{\frak{M}( \Omega_R\sqcup \Omega_S, X)}
\nc{\rswxs}{\frak{M}^\star( \Omega_R\sqcup \Omega_S, X)}
\nc{\Dl}{\leq_{_{{\rm Dl}}}} \nc{\Dll}{<_{_{{\rm Dl}}}} \nc{\bbs}{\mathbb{S}}
\nc{\orbsa}{$\Omega$-Rota-Baxter system\xspace}
\nc{\orbsas}{$\Omega$-Rota-Baxter systems\xspace}
\nc{\mrbs}{matching Rota-Baxter system\xspace}
\nc{\mrbss}{matching Rota-Baxter systems\xspace}

\nc\prbsla[4]{{R}_{#1}\left(#3\right)R_{#2}\left(#4\right)}
\nc\prbsra[4]{{R}_{#1\rightarrow#2}\left(R_{#1\rhd#2}\left(#3\right)#4\right)+{R}_{#1\leftarrow#2}\left(#3S_{#1\lhd#2}\left(#4\right)\right)}
\nc\prbslb[4]{{S}_{#1}\left(#3\right)S_{#2}\left(#4\right)}
\nc\prbsrb[4]{{S}_{#1\rightarrow#2}\left(R_{#1\rhd#2}\left(#3\right)#4\right)+{S}_{#1\leftarrow#2}\left(#3S_{#1\lhd#2}\left(#4\right)\right)}

\nc\rbsla[4]{\lc #3 \rc ^{R}_{#1} \lc #4 \rc ^R_{#2}}
\nc\rbslb[4]{\lc #3\rc ^{S}_{#1}  \lc #4 \rc ^S_{#2}}
\nc\rbsray[4]{\lc \lc #3 \rc^R_{#1\rhd#2} #4 \rc ^{R}_{#1\rightarrow#2}}
\nc\rbsraz[4]{\lc #3 \lc #4\rc^S_{#1\lhd#2}\rc ^{R}_{#1\leftarrow#2}}
\nc\rbsrby[4]{\lc \lc #3\rc ^R_{#1\rhd#2}#4\rc ^{S}_{#1\rightarrow#2}}
\nc\rbsrbz[4]{\lc #3 \lc #4 \rc ^S_{#1\lhd#2}\rc ^{S}_{#1\leftarrow#2}}
\nc\rbsrac[4]{\lc #3 \lc #4\rc^R_{#1\lhd#2}\rc ^{R}_{#1\leftarrow#2}}

\nc\rbslq[4]{\lc #3 \rc ^{Q}_{#1} \lc #4 \rc ^Q_{#2}}
\nc\rbslt[4]{\lc #3\rc ^{T}_{#1}  \lc #4 \rc ^T_{#2}}
\nc\rbsrqy[4]{\lc \lc #3 \rc^R_{#1\rhd#2} #4 \rc ^{Q}_{#1\rightarrow#2}}
\nc\rbsrqz[4]{\lc #3 \lc #4\rc^S_{#1\lhd#2}\rc ^{Q}_{#1\leftarrow#2}}
\nc\rbsrty[4]{\lc \lc #3\rc ^R_{#1\rhd#2}#4\rc ^{T}_{#1\rightarrow#2}}
\nc\rbsrtz[4]{\lc #3 \lc #4 \rc ^S_{#1\lhd#2}\rc ^{T}_{#1\leftarrow#2}}

\nc{\obr}[1]{\lc #1 \rc_\omega^R} \nc{\obs}[1]{\lc #1 \rc_\omega^S} \nc{\obq}[1]{\lc #1 \rc_\omega^*}
\nc{\obqa}[1]{\lc #1 \rc_\alpha^*} \nc{\obqb}[1]{\lc #1 \rc_\beta^*}
\nc{\obra}[1]{\lc #1 \rc_\alpha^R} \nc{\obrb}[1]{\lc #1 \rc_\beta^R}
\nc{\obsa}[1]{\lc #1 \rc_\alpha^S} \nc{\obsb}[1]{\lc #1 \rc_\beta^S}

\nc\nearro{\nearrow}
\nc\nwarro{\nwarrow}
\nc\searro{\searrow}
\nc\swarro{\swarrow}

\nc\nearrod[2]{#1\succ d(#2)}
\nc\nwarrod[2]{ #1\prec d(#2)}
\nc\searrod[2]{d(#1)\succ  #2}
\nc\swarrod[2]{d(#1)\prec #2}

\nc\nd{Novikov-dendriform\xspace}
\nc\ntd{Novikov-tridendriform\xspace}

\nc\vw{\diamondsuit}
\nc\ve{\vee}
\nc\wedg{\wedge}
\nc\ved[2]{d(#1) {\bullet} #2}
\nc\wedgd[2]{#1 {\bullet} d(#2)}

\begin{document}

\title[Differential ($q$-tri)dendriform algebras]{Weighted differential ($q$-tri)dendriform algebras}

\author{Yuanyuan Zhang}
\address{School of Mathematics and Statistics, Henan University, Henan, Kaifeng 475004, China}
\email{zhangyy17@henu.edu.cn}

\author{Huhu Zhang} \address{School of Mathematics and Statistics,
Lanzhou University, Lanzhou, 730000, P.\,R. China}
\email{zhanghh20@lzu.edu.cn}

\author{Tingzeng Wu}
\address{School of Mathematics and Statistics, Qinghai Nationalities University, Xining, Qinghai 810007,China;Qinghai Institute of Applied Mathematics, Xining, Qinghai 810007, P.R. China}
\email{mathtzwu@163.com}

\author{Xing Gao$^{*}$}\thanks{*Corresponding author}
\address{School of Mathematics and Statistics, Lanzhou University
Lanzhou, 730000, China;
School of Mathematics and Statistics
Qinghai Nationalities University, Xining, 810007, China;
Gansu Provincial Research Center for Basic Disciplines of Mathematics
and Statistics, Lanzhou, 730070, China
}
\email{gaoxing@lzu.edu.cn}

\date{\today}

\begin{abstract}
In this paper, we first introduce a weighted derivation on algebras over an operad $\cal P$, and prove that for the free $\cal P$-algebra, its weighted derivation is determined by the restriction on the generators.
As applications, we propose the concept of weighted differential ($q$-tri)dendriform algebras and study some basic properties of them.
Then Novikov-(tri)dendriform algebras are initiated, which can be induced from differential ($q$-tri) dendriform of weight zero. Finally, the corresponding free objects are constructed, in both the commutative and noncommutative contexts.
\end{abstract}

\makeatletter
\@namedef{subjclassname@2020}{\textup{2020} Mathematics Subject Classification}
\makeatother
\subjclass[2020]{
16W99, 
16S10, 
13P10, 
08B20 
}

\keywords{Derivation; differential ($q$-tri)dendriform algebra; Novikov algebra; Koszul dual operad}

\maketitle

\tableofcontents

\setcounter{section}{0}

\allowdisplaybreaks

\section{Introduction}
\subsection{(Weighted) differential algebras}
A differential algebra is an associative algebra equipped with a linear operator satisfying the Leibniz rule. The study of differential algebras is an algebraic approach to differential equations replacing analytic notions like differential quotients by abstract operations, initiated by Ritt~\cite{Ritt34, Ritt50} in 1930s and developed by Kolchin and his school~\cite{Kol73}.  After the fundamental work of Ritt and Kolchin, differential algebra has evolved into a vast area of mathematics that is important in both theory and applications, including differential Galois theory, differential algebraic geometry, differential algebraic groups~\cite{Kol73, Mag,SP03} and mechanic theorem proving~\cite{Wu84, Wu85}.
Guo and Keigher posed the concept of weighted differential algebra, including the usual differential algebra of weight zero and the difference algebra of weight one~\mcite{GK08} as examples. A more extensive notion of differential type algebras was studied in~\mcite{GSZ} in terms of the methods of Gr\"obner-Shirshov bases and rewriting systems.

Free objects are the most significant objects in a category.
It is well-known that the free differential algebra on a set is the polynomial algebra on the differential variables. This method can be generalized to the weighted case~\mcite{GK08}, just replacing the Leibniz rule by the weighted Leibniz rule. In general, for a multilinear variety Var of algebras (not necessary associative) with one binary operation, the free differential Var-algebra of weight zero generated by a set $X$ coincides with the free algebra Var$\langle \Delta X\rangle$ on the set $\Delta X$ of differential variables~\mcite{KSO}.

The more general notation of the free differential algebra on an algebra is more tricky. In this direction, Guo et al. recently showed that a Gr\"obner-Shirshov basis property of the base algebra can be extended to the free weighted differential algebra on this base algebra~\mcite{LG}. Zhou et al.~\mcite{QQWZ, QQZ} considered the free differential type algebra on an algebra. In more details, let $A$ be an algebra with a Gr\"obner-Shirshov basis for it and $B$ the free $\Phi$-algebra on $A$ with $\Phi$ being most of the differential type operator polynomial identities in~\cite{GSZ}. Zhou et al. provided a Gr\"obner-Shirshov basis and thereby a linear basis for $B$, in terms of some new monomial orders.

\subsection{($q$-tri)dendriform algebras}
Dendriform algebras were introduced by Loday~\cite{Lod93} in 1995 with motivation from algebraic $K$-theory. They have been studied quite extensively with connections to several areas in mathematics and physics, including operads ~\cite{Loday2}, homology~\cite{Fra2}, arithmetics~\cite{Loday3}, Hopf algebras~\cite{LoRo98,LoRo04,Ron02} and quantum field theory ~\cite{EFMP,Foissy}, in connection with the theory of renormalization of Connes and Kreimer~\cite{CK,CK1,CK2}.
Some years later after~\cite{Lod93}, Loday and Ronco introduced the concept of tridendriform algebra in the study of polytopes and Koszul duality~\cite{LoRo04}. The construction of free (tri)dendriform algebras can be referred to~\cite{Lod93,LoRo98,LoRo04}. Burgunder and Ronco~\cite{BR} proposed the concept of $q$-tridendriform algebras, which can simultaneously deal with tridendriform algebras~\cite{LoRo04} (when $q=1$) and $\mathcal{K}$-algebras~\cite{Chap} (when $q=0$). Recently, the study of a family version of (tri)dendriform algebras was given in~\cite{ZG, ZGM}, focusing on the constructions of free objects in both the commutative and noncommutative settings. Later, the relationship between free Rota-Baxter family algebras and free (tri)dendriform family algebras was characterized in~\cite{ZGM23}.

\subsection{The interaction of Rota-Baxter operator and differential operator}
A Novikov algebra is a special kind of a pre-Lie algebra, or left-symmetric algebra, arising in many contexts in mathematics and physics. The definition of a Novikov algebra appeared in~\cite{BN} devoted to the study of Poisson brackets of hydrodynamic type and the term Novikov algebra was proposed in~\cite{Os}.
Novikov algebras were studied by Zelmanov~\cite{Zelmanov} towards abstract algebraic structures, including classification and demonstrating that simple finite-dimensional Novikov algebras of characteristic zero are one-dimensional. A generic example of a Novikov algebra can be constructed from a commutative associative algebra $A$ with a derivation $d$ by setting $a \circ b := ad(b)$ on the same module $A$. Recently, noncommutative Novikov algebras were considered in~\cite{SK}, obtained in the same way from not necessarily commutative but associative algebras with a derivation.

It is known that a Rota-Baxter operator of weight zero on an associative algebra gives a dendriform
algebra~\cite{A04}. A generalized  process of a Rota-Baxter operator of weight zero on a $\mathcal{P}$-algebra giving a pre-$\mathcal{P}$-algebra has already been considered from the operadic point of view in~\cite{BGN}. As mentioned above, a derivation on a commutative associative algebra induces a Novikov algebra~\cite{BN}. Such a process is generalized so that a derivation on a $\mathcal{P}$-algebra induces a $Nov \,\tiny{\textcircled{}}\,\mathcal{P}$ algebra~\cite{KSO}. Here $Nov$ is the operad of Novikov algebras and $\tiny{\textcircled{}}$  stands for the Manin white product of operads~\cite{GK}.

A picture is worth a thousand words. Thanks to the results in Section~\ref{sect:2} and some well-known facts, algebras mentioned above and focused in the present paper can be gathered in the following diagram
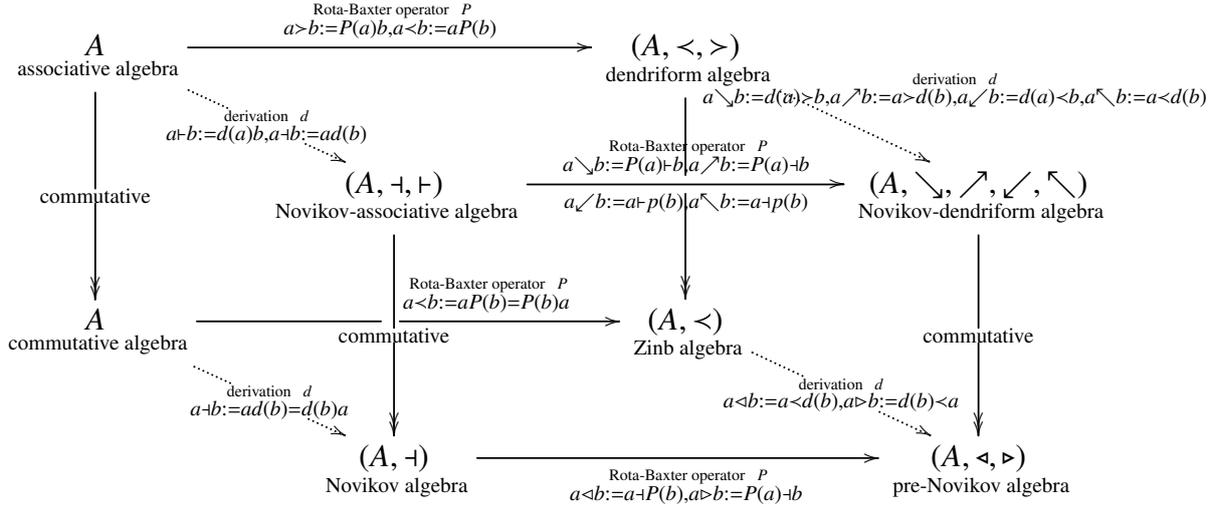
\nc\asa{associative\xspace}
\nc\al{algebra\xspace}
\nc\dal{dendriform\xspace}
\[\xymatrix{
    \underset{\text{ \asa \al }}{A}&& \underset{\text{ \dal \al }}{(A, \prec,\succ)}\\
    &\underset{\text{ Novikov-\asa \al }}{(A, {\dashv},{\vdash})}&&\underset{\text{ Novikov-\dal \al }}{(A, \searrow, \nearrow, \swarrow, \nwarrow)}\\
    \underset{\text{ commutative \al }}{A}&&\underset{\text{ Zinb \al }}{(A, \prec)}\\
    &\underset{\text{ Novikov \al }}{(A, {\dashv})}&&\underset{\text{ pre-Novikov \al }}{(A, \triangleleft,\triangleright)}\\
    \ar@{.>}"1,1";"2,2"|-{\overset{\text{derivation }~d}{ a{\vdash} b:=d(a)b, a{\dashv} b:=ad(b) }}
    \ar@{.>}"1,3";"2,4"^{\overset{\text{derivation }~d}{ a\searrow b:=d(a)\succ b, a\nearrow b:=a\succ d(b), a\swarrow b:=d(a)\prec b, a\nwarrow b:=a\prec d(b) }}
    \ar@{.>}"3,1";"4,2"|-{\overset{\text{derivation }~d}{ a {\dashv} b:=ad(b)=d(b)a }}
    \ar@{.>}"3,3";"4,4"|-{\overset{\text{derivation }~d}{ a \lhd b:=a\prec d(b) , a \rhd b:=d(b)\prec a}}
    \ar"1,1";"1,3"^{\overset{\text{ Rota-Baxter operator }~P}{{a\succ b:=P(a)b, a\prec b:=aP(b) }}}
    \ar"2,2";"2,4"^{\overset{\text{ Rota-Baxter operator }~P}{{a\searrow b:=P(a){\vdash}b, a\nearrow b:=P(a){\dashv}b }}}_{a\swarrow b:=a{\vdash}p(b), a\nwarrow b:=a{\dashv}p(b)}
    \ar"4,2";"4,4"_{\overset{\text{ Rota-Baxter operator }~P}{{a\lhd b:=a {\dashv} P(b), a\rhd b:=P(a){\dashv}b }}}
    \ar@{->>}"1,1";"3,1"|-{\text{ commutative }}
    \ar@{->>}"2,2";"4,2"|-{\text{ commutative }}
    \ar@{->>}"2,4";"4,4"|-{\text{ commutative }}
    \ar"3,1";"3,3"|\hole^{\qquad\qquad\qquad\overset{\text{ Rota-Baxter operator }~P}{{ a\prec b:=aP(b)=P(b)a }}}
    \ar@{->>}"1,3";"3,3"|\hole
}\]
Here, $``\twoheadrightarrow$'' stands for ``commutative'', $``\longrightarrow$'' stands for ``adding a Rota-Baxter operator $P$ of weight zero'',
$``\dashrightarrow$'' stands for ``adding a derivation $d$ of weight zero''. If $``\longrightarrow$'' stands for ``adding a Rota-Baxter operator $P$ of weight $q$'', then right view of the cube above is replaced by
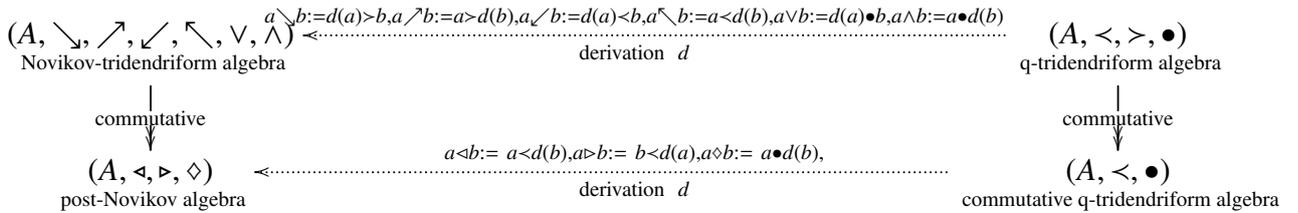
\[\xymatrix{
    \underset{\text{ Novikov-tridendriform \al }}{(A, \searrow, \nearrow, \swarrow, \nwarrow, \vee,\wedge)}
    &&&&&&&&\underset{\text{ q-tridendriform \al }}{(A, \prec,\succ, \bullet)}\\
    \underset{\text{ post-Novikov \al }}{(A, \triangleleft,\triangleright,\vw)}
    &&&&&&&&\underset{\text{ commutative q-tridendriform \al }}{(A, \prec, \bullet)}\\
    \ar@{.>}"1,9";"1,1"^{\text{derivation }~d}_{ a\searrow b:=d(a)\succ b, a\nearrow b:=a\succ d(b), a\swarrow b:=d(a)\prec b, a\nwarrow b:=a\prec d(b),
    a\vee b:=d(a)\bullet b, a\wedge b:=a\bullet d(b) }
    \ar@{.>}"2,9";"2,1"^{\text{derivation }~d}_{a \lhd b :=~ \nwarrod{a}{b}, a \rhd b:=~ \nwarrod{b}{a},a \vw b: = ~\wedgd{a}{b},}
    \ar@{->>}"1,1";"2,1"|-{\text{ commutative }}
    \ar@{->>}"1,9";"2,9"|-{\text{ commutative }}
}\]
If the initial point is an associative algebra $A$ with a Rota-Baxter operator $P$ of weight zero and a derivation $d$ of weight zero such that $dP=Pd$,
then the above cube commutes. The condition $dP=Pd$ is pivotal. See Proposition~\mref{prop:ddrb} for details.
If $dP\neq Pd$, then the top view and the bottom view of the above cube are not commutative.

\subsection{Main results and outline of the paper}
The relationship between dendriform algebras and derivations have been studied separately until the appearance of the paper~\cite{LLB22}. There the authors initiated the concept of a differential dendriform algebra, which is a
dendriform algebra equipped with a finite set of commutative derivations.

In the present paper, we focus on one derivation and
initiate the concept of a weighted derivation on an algebra over an operad. As applications, we generalize the notation of a differential dendriform algebra to the case of a weighted differential ($q$-tri)dendriform algebra, with some basic properties of them exposed.
Two new algebras, named Novikov-(tri)dendriform algebras, are induced from differential ($q$-tri)dendriform algebras of weight zero, respectively. Further, if the differential
($q$-tri)dendriform algebra of weight zero is commutative, then we obtain the pre(post)-Novikov algebra. Also we characterize the Koszul dual operad of the operad of differential ($q$-tri)dendriform algebras of weight
zero.
The corresponding free objects in both the commutative and noncommutative contexts are constructed, respectively. It is worth highlighting that our results are not under the framework of~\mcite{KSO}, as our derivation is of weight $\lambda$ and the ($q$-tri)dendriform algebra has more than one binary operation.

The paper is organized as follows. In Section~\ref{sec:freediffpalg}, we introduce the concept of a weighted derivation on an algebra over an operad $\cal P$.
For free $\cal P$-algebras, we prove that the weighted derivation is determined by its restriction on generators (Proposition~\ref{prop:wdiff}). Section~\ref{sect:2} is devoted to propose the concepts of weighted differential ($q$-tri)dendriform algebras and to give some basic properties (Propositions~\mref{prop:ddrb} and~\mref{pp:dualdiffdend}).
A (commutative) differential dendriform algebra of weight zero induces a (pre-Novikov) Novikov-dendriform algebra (Propositions~\mref{prop:dendnd} and~\mref{prop:cdd}).
Parallelly, a (commutative) differential $q$-tridendriform algebra of weight zero induces a (post-Novikov) Novikov-tridendriform algebra (Propositions~\mref{prop:tdendnd} and~\mref{prop:ctdendnd}).

In Section~\ref{sec:difftridend}, we construct the free weighted commutative differential $q$-tridendriform algebra by the weighted quasi-shuffle product (Theorem~\mref{thm:difftridend}). As an application, we obtain the free weighted commutative differential dendriform algebra via the shuffle-product (Theorem~\mref{coro:freecddend}). Section~\ref{sec:freedifftridend} is devoted to build the free differential $q$-tridendriform algebra of weight $\lambda$ on top of valently decorated Schr\"oder trees (Theorem~\mref{thm:freedifftridendriform}). We focus on the construction of the free differential dendriform algebra of weight $\lambda$ in terms of planar binary trees (Theorem~\mref{thm:freediffdendriform}).

\smallskip
{\bf Notation.}
Throughout this paper, we fix a commutative unitary ring $\bfk$,
which will be the base ring of all modules, algebras, tensor products, as well as linear maps. By an algebra we mean a unitary associative noncommutative algebra, unless the contrary is specified. Denote by $S_n$ the symmetric group of degree $n$ with $n\geq 1$ and by $\mathbb{N}$ the set of nonnegative integers.

\section{Weighted derivations on operadic algebras}
\mlabel{sec:freediffpalg}
In this section, we first recall some basic concepts of operads~\cite{LV} and define a weighted derivation on a $\cal  P$-algebra over an operad $\cal P$. Then we prove that the weighted derivation on the free $\cal P$-algebra is determined by its restriction on the generators.

\begin{defn}\cite[Section 5.2.1]{LV}\label{defn:operad}
An {\bf operad}  is  a collection $\cal P=\{\cal P(n)\}_{n\geq0}$ of right $S_n$-modules endowed with two collections of morphisms of $S_n$-modules
$$\gamma=\{\gamma_n\}_{n\geq0}: \cal P\circ \cal P\to \cal P\,\text{ and }\,
\eta=\{\eta_n\}_{n\geq0}:  I\to \cal P,$$
where $\cal P\circ \cal P=\bigoplus_{n\geq0} \cal P (n)\otimes_{S_n} \cal P^{\otimes n}$ and $I=\{I(0), I(1), I(2),\cdots\}=\{0,\bfk, 0,\cdots\},$ such that the following diagrams commute
$$
\xymatrix{
& \cal P\circ(\cal P\circ \cal P) \ar[r]^-{\Id\circ\gamma} & \cal P\circ  \cal P \ar[dd]^{\gamma} &\\
(\cal P\circ \cal P)\circ  \cal P \ar[ur]^= \ar[d]_{\gamma\circ \Id} & & &
I\circ\cal P \ar[dr]_= \ar[r]^-{\eta\circ\Id} & \cal P\circ \cal P \ar[d]^{\gamma} &\ar[l]_-{\Id\circ\eta}\cal P\circ I \ar[dl]^= \\
\cal P\circ  \cal P \ar[rr]^{\gamma}         &         &  \cal P &  & \cal P.&
}
$$
\end{defn}
In this paper, we consider only reduced operads, i.e. such that $\cal P(0)=\{0\}$.

\begin{defn}\cite[Section 5.2.3]{LV}
Let $\cal P$ be an operad. A {\bf $\cal P$-algebra} is a module $A$ equipped with a linear map $\gamma_A:\cal{P}(A)\ra A$ such that the following diagrams commute:
$$
\xymatrix{
& \cal P(\cal P(A)) \ar[r]^-{\cal P(\gamma_A)} & \cal P(A) \ar[dd]^{\gamma_A} \\
(\cal P\circ \cal P)(A) \ar[ur]^= \ar[d]_{\gamma(A)} & & & I(A)\ar[dr]_= \ar[r]^-{\eta(A)} & \cal P(A) \ar[d]^{\gamma(A)}\\
\cal P(A) \ar[rr]^{\gamma_A}         &         & A &  & A.
}
$$
Here $\cal P(A)=\sum_{n\geq 0} \cal P(n)\otimes_{S_n} A^{\otimes n}.$
\end{defn}

For any module $V$, $\cal{P}(V)$ is a $\cal P$-algebra by defining
$\gamma(V):=\gamma_{\cal{P}(V)}: \cal P(\cal P(V))\ra \cal{P}(V).$

\begin{prop}\cite[Proposition 5.2.6]{LV}
The $\cal P$-algebra $(\cal{P}(V), \gamma(V))$, equipped with the embedding $j_V:V\hookrightarrow \cal{P}(V)$, is the free $\cal  P$-algebra on a module $V$.
In more details, for any $\cal P$-algebra $A$ and any linear map $f:V\ra A$, there is a unique $\cal P$-algebra morphism $\lbar{f}$ such that $f = \lbar{f}j_V$.
\end{prop}

In a heuristic way, let us recall that~\cite{GK08} a derivation $d$ of weight $\lambda$ on an associative algebra $A$ is a linear operator $d:A\rightarrow A$ such that
$$d(ab)=d(a)b+ ad(b)+\lambda d(a)d(b), \quad a,b\in A.$$
This sheds light on the following general concept.

\begin{defn}
Let $\cal P$ be an operad, $A$ be a $\cal P$-algebra and $\lambda\in \bfk$.
A linear map $d:A\ra A$ is called a {\bf derivation} of weight $\lambda$ on $A$ if
it makes the following diagram commutative
$$
\xymatrix{
\cal P( A)  = \cal P\circ A \ar[r]^-{\gamma_A}\ar[d]_{\id_{\cal P}\hat{\circ} d} & A \ar[d]^{d}\\
\cal P( A) = \cal P\circ A\ar[r]^-{\gamma_A} & A.
}
$$
Here $\id_{\cal P}\hat{\circ} d :=\{(\id_{\cal P}\hat{\circ} d)_n\}_{n\geq 0}$ and
$$(\id_{\cal P}\hat{\circ} d)_n :=\sum_{k=1}^n\lambda^{k-1}\bigg(\sum_{1\leq i_1<\cdots<i_k\leq n}
\id_{\cal P}\otimes \Big(\id_A\otimes\cdots\otimes\id_ A\otimes \underbrace{d}_{i_1\text{-th position}}\otimes \id_A\otimes \cdots \otimes \id_A\otimes \underbrace{d}_{i_k\text{-th position}}\otimes\id_A\otimes\cdots \otimes \id_A \Big)\bigg).$$
\mlabel{defn:dpo}
\end{defn}

Denote by $\rm{Diff}_{\lambda}(A)$ the set of derivations of weight $\lambda$ on a $\cal P$-algebra $A$.
The following result is a generalization of~\cite[Proposition~6.3.6]{LV}.

\begin{prop}
\mlabel{prop:wdiff}
Let $(\cal P, \gamma, \eta)$ be an operad, $V$ a module and $\lambda\in \bfk$. Then any derivation $d$ of weight $\lambda$ on the free $\cal P$-algebra $\cal P(V)$ is completely characterized by its restriction on the generators $V\ra \cal P(V):$
\[
\rm{Diff}_\lambda(\cal P(V))\cong \rm{Hom}(V, \cal P(V)).
\]
In more details, given a linear map $\phi: V\ra \cal{P}(V)$, the unique derivation $d_\phi: \cal{P}(V) \rightarrow \cal{P}(V)$ of weight $\lambda$  extending $\phi$ is determined by
\[
d_\phi:=(\gamma \circ \id_V) (\id_{\cal P}\hat{\circ} \phi): \,
{\cal P} \circ V \rightarrow {\cal P} \circ {\cal P}(V) = {\cal P}\circ {\cal P}\circ V \rightarrow {\cal P}\circ V.
\]
that is, for any $f\in \cal P(n)$ and any $v_1,\cdots,v_n\in V$ with $n\geq 0$,
\begin{eqnarray*}\label{eq:inductiondiff}
d_\phi(f(v_1,\cdots,v_n)):=
\sum_{k=1}^{n} \lambda^{k-1}\bigg(\sum_{1\leq i_1<\cdots< i_k \leq n}
f\Big(v_1,\cdots, \phi(v_{i_1}),\cdots,\phi(v_{i_2}),\cdots,\phi(v_{i_k}),\cdots,v_n \Big)\bigg).
\end{eqnarray*}
\end{prop}

\begin{proof}
For any $f\in \cal P(n)$ and any $v_1,\cdots,v_n\in V$, we have
\begin{eqnarray*}
&&d_\phi(f(v_1,\cdots,v_n))\\
&=&(\gamma\circ \id_V)(\id_{\cal P}\hat{\circ} d_\phi)_n(f, v_1,\cdots,v_n)\\
&=&\sum_{k=1}^n\lambda^{k-1}(\gamma \circ \id_V)\bigg(\sum_{1\leq i_1<\cdots<i_k\leq n}
\id_{\cal P} \otimes\Big(\id_{\cal P(V)}\otimes\cdots\otimes\id_{\cal P(V)}\otimes \underbrace{d_\phi}_{i_1\text{-th position}}\otimes \id_{\cal P(V)}\otimes \cdots\\
&&\otimes \id_{\cal P(V)}\otimes \underbrace{d_\phi}_{i_k\text{-th position}}\otimes\id_{\cal P(V)}\otimes\cdots \otimes \id_{\cal P(V)} \Big)\bigg)(f, v_1,\cdots,v_n)\\
&=& \sum_{k=1}^{n} \lambda^{k-1}\bigg(\sum_{1\leq i_1<\cdots< i_k \leq n}f(v_1,\cdots, d_\phi(v_{i_1}),\cdots,d_\phi(v_{i_2}),\cdots,d_\phi(v_{i_k}),\cdots,v_n)\bigg)\\
&=&\sum_{k=1}^{n} \lambda^{k-1}\bigg(\sum_{1\leq i_1<\cdots< i_k \leq n}f(v_1,\cdots, \phi(v_{i_1}),\cdots,\phi(v_{i_2}),\cdots,\phi(v_{i_k}),\cdots,v_n)\bigg),
\end{eqnarray*}
which is fixed. This completes the proof.
\end{proof}

\section{Weighted differential ($q$-tri)dendriform algebras and their Koszul dual operads} \label{sect:2}
In this section, as applications of the last section, we propose the concepts of
weighted differential ($q$-tri)dendriform algebras, and give some basic properties of them.
In particular, for the differential ($q$-tri)dendriform algebras of weight zero,
we obtain some non-associative algebras induced from them and compute their Koszul dual operads.

\subsection{Weighted differential ($q$-tri)dendriform algebras}
The concept of a dendriform algebra was introduced by Loday~\cite{Lod93} in 1995 with motivation from algebraic $K$-theory.

\begin{defn}\label{defn:dend}
A {\bf dendriform algebra} (previously also called a dendriform dialgebra) is a \bfk-module $D$ with two binary operations $\prec$ and $\succ$  on $D$ such that
 \begin{align*}
 (a \prec b) \prec c=&\ a \prec(b \prec c+b \succ c),\\
(a \succ b) \prec c=&\ a \succ(b \prec c), \\
 (a \prec b+a \succ b) \succ c=&\ a \succ(b \succ c),\quad \forall a, b, c \in D.
\end{align*}
A dendriform algebra $(D, \prec, \succ)$ is called {\bf commutative} if
$
a \succ b=b \prec a$ for $a, b\in D$.

\end{defn}

The concept of $q$-tridendriform algebras was initiated by Burgunder and Ronco~\cite{BR}.

\begin{defn}\cite[Definition 1.1]{BR}
Let $q\in \bfk$. A {\bf $q$-tridendriform algebra} is a \bfk-module $T$ together with three binary operations $\prec$, $\succ$ and $\bullet$ such that
\begin{align*}
\label{tridend1}
(a \prec b) \prec c=& \ a \prec( b \star_q c),\\
 (a \succ b) \prec c=& \ a \succ(b \prec c),\\
 (a\star_q b) \succ c=& \ a \succ(b \succ c),\\
(a \succ b) \bullet c=& \ a \succ(b \bullet c),\\
(a \prec b) \bullet c=& \ a \bullet(b \succ c),\\
 (a \bullet b) \prec c=& \ a \bullet(b \prec c),\\
(a \bullet b) \bullet c=& \ a \bullet(b \bullet c),\quad \forall a, b, c \in T,
\end{align*}
where $\star_q:= \prec +\succ + q \bullet $.
A $q$-tridendriform algebra $(T, \prec, \succ, \bullet)$ is called {\bf commutative} if
$a \succ b=b \prec a$ and $a \bullet b=b \bullet a$ for $a, b \in T.$
\mlabel{defn:qtrid}
\end{defn}

We collect the following facts.
\begin{remark}\cite{BR}\mlabel{re:qtrid}
\begin{enumerate}
\item The operation $\star_q=\, \prec+ \succ + q\bullet$ is associative.

\item A $q$-tridendriform algebra $(T,\prec,\succ,\bullet)$ induces a dendriform algebra $(T,\prec,\succ_q)$ by $\succ_q:= \succ + q\bullet$.

\item If $q=1$, a $q$-tridendriform algebra reduces to a tridendriform algebra given in~\cite{LoRo04}.

\item If $q=0$, a $q$-tridendriform algebra reduces to a (ungraded) $\mathcal{K}$-algebra described in~\cite{Chap}.

\item \mlabel{it:rb2qtri}
A Rota-Baxter algebra $(A, \cdot, P)$ of weight $q$ induces a $q$-tridendriform algebra $(A, \prec, \succ, \cdot)$ by
\[a \prec_P b:=a \cdot P(b)\,\text{ and }\, a \succ_P b:=P(a) \cdot b,\quad\forall a, b \in A.\]
\end{enumerate}
\end{remark}
Quite recently, the concept of a differential dendriform algebra was proposed in~\cite[Definition 4.22]{LLB22}.
Motivated by this, we introduce the weighted version with one derivation,
also as an example of Definition~\mref{defn:dpo}.

\begin{defn}
Let $\lambda \in \bfk$ and $(D,\prec,\succ)$ be a dendriform algebra.
A linear map $d:D\rightarrow D$ is called a {\bf derivation of weight $\lambda$} on $D$ if
\begin{align}
 d(a\prec b)=&\ d(a)\prec b+a\prec d(b)+\lambda d(a)\prec d(b),\mlabel{eq:diffdend2} \\
 d(a\succ b)=&\ d(a)\succ b+a\succ d(b)+\lambda d(a)\succ d(b),\quad \forall a, b \in D.  \mlabel{eq:diffdend3}
\end{align}
In this case, we call the quadruple $(D,\prec,\succ,d)$ a {\bf differential dendriform algebra of weight $\lambda$}.
Further, it is called {\bf commutative} if the dendriform algebra $(D,\prec,\succ)$ is commutative.
\end{defn}

Analogously, we propose the following concept as another example of Definition~\mref{defn:dpo}.

\begin{defn}~\label{defn:diffqtri}
Let $\lambda, q \in \bfk$ and $(T,\prec, \succ, \bullet)$ be a $q$-tridendriform algebra. A linear map $d:T\rightarrow T$ is called a {\bf derivation of weight $\lambda$} on $T$ if
\begin{align}
 d(a\prec b)=& \ d(a)\prec b+a\prec d(b)+\lambda d(a)\prec d(b),\mlabel{eq:difftri2}\\
 d(a\succ b)=& \ d(a)\succ b+a\succ d(b)+\lambda d(a)\succ d(b),  \mlabel{eq:difftri4}\\
 d(a\bullet b)=& \ d(a)\bullet b+a\bullet d(b)+\lambda d(a)\bullet d(b),\quad \forall a, b \in T. \mlabel{eq:difftri3}
\end{align}
In this case, we call the quintuple $(T,\prec,\succ, \bullet, d)$ {\bf a differential $q$-tridendriform algebra of weight $\lambda$}.
Moreover, it is named {\bf commutative} if the $q$-tridendriform algebra $(T,\prec,\succ, \bullet)$ is commutative.
\end{defn}

\begin{remark}
For a commutative differential dendriform algebra $(D,\prec,\succ, d)$  of weight $\lambda$, Eq.~(\mref{eq:diffdend2}) is equivalent to
Eq.~(\mref{eq:diffdend3}). For a commutative differential $q$-tridendriform algebra $(T,\prec,\succ, \bullet, d)$  of weight $\lambda$, Eq.~(\mref{eq:difftri2}) is exactly Eq.~(\mref{eq:difftri4}), and the triple $(T, \bullet, d)$ is a differential algebra of weight $\lambda$.
\end{remark}

As usual, we have the following fact.

\begin{prop}\label{prop:diffdend}
Let $\lambda, q\in\bfk$. Let $(A, \cdot, d)$ be a differential algebra of weight $\lambda$ and $P$ be a Rota-Baxter operator of weight $q$ on $(A, \cdot)$ such that $d P=P d$.
\begin{enumerate}
\item If $q=0$, then the quadruple $(A, \prec_P, \succ_P, d)$ is a differential dendriform algebra of weight $\lambda$, where
$$a\prec_P b := a P(b),\quad a\succ_P b := P(a) b, \quad\forall a, b\in A.$$
\mlabel{item:drbd}
\item If $q\neq0$, then the quintuple $(A, \prec_P, \succ_P, \cdot, d)$ is a differential $q$-tridendriform algebra of weight $\lambda$, where
$$a\prec_P b := a P(b),\quad a\succ_P b := P(a) b, \quad \forall a, b\in A.$$ \mlabel{item:drbt}

\item If $q\neq0$, then the quadruple $(A, \prec_P, \succ_P, d)$ is a differential dendriform algebra of weight $\lambda$, where
$$a\prec_P b := a P(b)+ qab, \quad a\succ_P b := P(a) b, \quad\forall a, b\in A$$
or
$$a\prec_P b := a P(b), \quad a\succ_P b := P(a) b+ qab, \quad\forall a, b\in A.$$
\mlabel{item:drbdq}
\end{enumerate}
\mlabel{prop:ddrb}
\end{prop}

\begin{proof}
We just verify Item~\mref{item:drbt}. The Items~\mref{item:drbd} and~\mref{item:drbdq} can be proved in {the} same way.
By Remark~\mref{re:qtrid}~\mref{it:rb2qtri}, the quadruple $(A, \prec_P, \succ_P, \cdot)$ is a $q$-tridendriform algebra. {For any $a,b\in A$,} Eqs.~\eqref{eq:difftri2} and~\eqref{eq:difftri4} follow from
{\allowdisplaybreaks
\begin{align*}
d(a\prec_P b)=&\ d(aP(b))\\
=&\ d(a) P(b)+a d(P(b))+\lambda d(a)d(P(b))\\
=&\ d(a)P(b)+aP (d(b))+\lambda d(a)P( d(b))\\
=&\ d(a)\prec_P b+a\prec_P d(b)+\lambda d(a)\prec_P d(b),\\
d(a\succ_P b)=&\ d(P(a)b)\\
=&\ d(P(a))b+P(a) d(b)+\lambda d(P(a))d(b)\\
=&\ P(d(a))b+P(a)d(b)+\lambda P(d(a)) d(b)\\
=&\ d(a)\succ_P b+a\succ_P d(b)+\lambda d(a)\succ_P d(b),
%
\end{align*}}
as required.
\end{proof}

\subsection{ Novikov-(tri)dendriform algebras and pre(post)-Novikov algebras}
We are going to induce some non-associative algebras from differential ($q$-tri)dendriform algebras of weight zero. Let us begin with the concept of Novikov algebras.

\begin{defn}~\cite{Os}
A {\bf (left) Novikov algebra} is a \bfk-module $N$ with a binary operation $\circ$ on $N$ such that
\begin{eqnarray*}
(a \circ b)\circ c&=&(a \circ c)\circ b,\\
(a\circ b)\circ c-a\circ (b\circ c)&=&(b\circ a)\circ c-b\circ (a\circ c),\quad\forall a,b,c\in N.
\end{eqnarray*}
\end{defn}

\begin{defn}
A {\bf Novikov-associative algebra} is a \bfk-module $A$ together with two binary operations $\vdash$ and $\dashv$ on $A$ such that
\begin{eqnarray}
(a \vdash b)\dashv c&=&a\vdash (b\dashv c),\label{eq:na1}\\
(a\dashv b)\dashv c-a\dashv (b\vdash c)&=&a\vdash (b\vdash c)-(a \dashv b)\vdash c,\quad \forall a,b,c\in A\label{eq:na2}.
\end{eqnarray}
A { Novikov-associative algebra} $(A, \vdash,\dashv)$ is called {\bf commutative } if
$a \dashv b= b\vdash a$ for $a,b\in A .$
\end{defn}

\begin{remark}
\begin{enumerate}
\item Denote by $\tiny{\textcircled{}}$ the Manin white product. A differential associative algebra of weight zero induces a
$Nov\,\tiny{\textcircled{}}\, Ass$-algebra~\cite[Theorem 7]{KSO}, where
$Nov$ (resp. $Ass$) is the operad of Novikov algebras (resp. associative algebras).
By~\cite[Lemma 2.3]{Lod10}, the $Nov\,\tiny{\textcircled{}}\, Ass$-algebra satisfies Eqs.~\eqref{eq:na1} and~\eqref{eq:na2}. So we adopt the name Novikov-associative algebra in this paper.
Such kind name will be used in Definitions~\mref{defn:nda} and~\mref{defn:nta}.
The Novikov-associative algebra is also named noncommutative Novikov algebras in~\mcite{SK}.

\item If the triple $(A, \vdash,\dashv)$  is a {commutative Novikov-associative algebra}, then
$(A, \dashv)$ is a Novikov algebra.  Conversely, if $(A, \circ)$ is a Novikov algebra, then $(A, \vdash,\dashv)$  is a {commutative Novikov-associative algebra}, where $a\dashv b= b\vdash a:=a\circ b.$
\end{enumerate}
\end{remark}

\begin{defn}\label{defn:nda}
A {{\bf \nd algebra}} is a \bfk-module $N$ together with four binary operations $\searrow$, $\nearrow$, $\swarrow$ and $\nwarrow$ on $N$ such that, for any $a,b,c\in N$,
\begin{eqnarray}
(a \swarro b)\nwarro c&=&a\swarro (b\nearro c+b \nwarro c), \label{eq:nd1}\\
(a \searro b)\nwarro c&=&a\searro (b\nwarro c),\\
(a\searro b+a\swarro b)\nearro c&=&a\searro (b\nearro c),\label{eq:nd3}\\
(a\nwarro b)\nwarro c-a\nwarro (b\searro c+b\swarro c)&=&a\swarro (b\searro c+b\swarro c)-(a \nwarro b)\swarro c,\label{eq:nd4}\\
(a\nearro b)\nwarro c-a\nearro (b\swarro c)&=&a\searro (b\swarro c)-(a \nearro b)\swarro c,\\
(a\nearro b+a\nwarro b)\nearro c-a\nearro (b\searro c)&=&a\searro (b\searro c)-(a\nearro b+a\nwarro b)\searro c.\label{eq:nd6}
\end{eqnarray}
A { \nd algebra} $(N, \searrow, \nearrow, \swarrow,\nwarrow)$ is called {\bf commutative } if
$$a \nwarro b= b\searro a \,\text{ and }\, a\nearro b =b\swarro a, \quad\forall a,b\in N .$$
\end{defn}

\begin{prop} \label{prop:dendnd}
A differential dendriform algebra $(D,\prec,\succ, d)$ of weight zero
induces a \nd algebra $(D, \searrow, \nearrow, \swarrow,\nwarrow)$,
where
\begin{eqnarray*}
&&a \searro b :=\searrod{a}{ b} , \quad a\nearro b :=\nearrod{a}{b}, \\
&&a \swarro b:= \swarrod{a}{b},\quad  a\nwarro b:= \nwarrod{a}{b}, \quad\forall a,b\in {D}.
\end{eqnarray*}
Moreover, if define
\begin{align*}
a\vdash b :=&\ a\searro b+a\swarro b=d(a)\succ b+ d(a)\prec b, \\
a\dashv b :=&\ a\nearro b+ a\nwarro b= a\succ d(b)+ a\prec d(b),\quad{\forall a,b,c\in D,}
\end{align*}
then the triple $(D, \vdash,\dashv)$ is a Novikov-associative algebra.
\end{prop}

\begin{proof}
{For any $a,b,c\in D,$} the Eqs.~\eqref{eq:nd1}-\eqref{eq:nd6} follow from
\begin{eqnarray*}
(a \swarro b)\nwarro c&=&\nwarrod{(\swarrod{a}{b})}{ c}\\
&=&\swarrod{a}{ (\nearrod{b}{ c}+\nwarrod{b}{ c})}\\
&=&a\swarro (b\nearro c+b \nwarro c),\\
(a \searro b)\nwarro c&=&\nwarrod{(\searrod{a}{  b})}{ c}\\
&=&\searrod{a}{ (\nwarrod{b}{ c})}\\
&=&a\searro (b\nwarro c),\\
(a\searro b+a\swarro b)\nearro c&=&\nearrod{(\searrod{a}{ b}+\swarrod{a}{ b})}{ c}\\
&=&\searrod{a}{ (\nearrod{b}{c})}\\
&=&a\searro (b\nearro c),
\end{eqnarray*}
and
\begin{eqnarray*}
&&(a\nwarro b)\nwarro c-a\nwarro (b\searro c+b\swarro c)\\
&=&\nwarrod{(\nwarrod{a}{ b})}{ c}-\nwarrod{a}{ (\searrod{b}{ c}+\swarrod{b}{ c})}\\
&=&\nwarrod{(\nwarrod{a}{ b})}{ c}-a\prec\big(d(d(b))\succ c+d(b)\succ d(c)+d(d(b))\prec c+d(b)\prec d(c)\big)\\
&=&{-a\prec\big(d(d(b))\succ c+d(d(b))\prec c\big)}\\
&=&-\big(a\prec d(d(b))\big)\prec c\\
&=&\swarrod{a}{ (\searrod{b}{ c}+\swarrod{b}{ c})}-\big(d(a)\prec d(b)+a\prec d(d(b))\big)\prec c\\
&=&\swarrod{a}{ (\searrod{b}{ c}+\swarrod{b}{ c})}-\swarrod{\nwarrod{a}{  b}}{ c}\\
&=&a\swarro (b\searro c+b\swarro c)-(a \nwarro b)\swarro c,\\
&&(a\nearro b)\nwarro c-a\nearro (b\swarro c)\\
&=&\nwarrod{(\nearrod{a}{ b})}{ c}-\nearrod{a}{ \swarrod{b}{c}}\\
&=&\nwarrod{(\nearrod{a}{ b})}{ c}-a\succ\big(d(d(b))\prec c+d(b)\prec d(c)\big)\\
&=&-a\succ\big(d(d(b))\prec c\big)\\
&=&-\big(a\succ d(d(b))\big)\prec c\\
&=&\searrod{a}{ (\swarrod{b}{ c})}-\big(d(a)\succ d(b)+a\succ d(d(b))\big)\prec c\\
&=&\searrod{a}{ (\swarrod{b}{ c})}-\swarrod{\nearrod{a}{  b}}{ c}\\
&=&a\searro (b\swarro c)-(a \nearro b)\swarro c,\\
&&(a\nearro b+a\nwarro b)\nearro c-a\nearro (b\searro c)\\
&=&\nearrod{(\nearrod{a}{ b}+\nwarrod{a}{ b})}{ c}-\nearrod{a}{ \searrod{b}{ c}}\\
&=&-a\succ(d(d(b))\succ c)\\
&=&-(a\succ d(d(b))+a\prec d(d(b)))\succ c\\
&=&\searrod{a}{ (\searrod{b}{ c})}-(\nearrod{a}{ b}+\searrod{\nwarrod{a}{ b})}{ c}\\
&=&a\searro (b\searro c)-(a\nearro b+a\nwarro b)\searro c.
\end{eqnarray*}
Further more, Eq.~\eqref{eq:na1} (resp. Eq.~\eqref{eq:na2}) is a direct result of the summations of Eqs.~\eqref{eq:nd1}-\eqref{eq:nd3} (resp. Eqs.~\eqref{eq:nd4}-\eqref{eq:nd6}).
\end{proof}

The following is the commutative version of Definition~\mref{defn:nda}.

\begin{defn}~\cite{HBG}\label{defn:prenov}
A {\bf pre-Novikov algebra} is a \bfk-module $N$ together with two binary operations $\lhd$ and $\rhd$ on $N$ such that
\begin{eqnarray}
(a\lhd b)\lhd c-a\lhd (b\rhd c+b\lhd c)&=&(b\rhd a)\lhd c-b\rhd (a\lhd  c), \label{eq:pn1}\\
(a\rhd b+a\lhd b)\rhd c-a\rhd (b\rhd c)&=&(b\rhd a+b\lhd a)\rhd c-b\rhd(a\rhd c),\label{eq:pn2}\\
(a\lhd b)\lhd c&=&(a\lhd c)\lhd b,\label{eq:pn3}\\
(a\rhd b+a\lhd b)\rhd c&=&(a\rhd c)\lhd b,\quad{\forall a, b,c\in N.}\label{eq:pn4}
\end{eqnarray}
\end{defn}
\begin{remark}\label{re:cnd}
If the quintuple $(N, \searrow, \nearrow, \swarrow,\nwarrow)$  is a {commutative \nd algebra}, then the triple $(N, \nwarrow,\nearrow)$ is a pre-Novikov algebra.  Conversely, if the triple $(N, \lhd,\rhd)$ is a pre-Novikov algebra, then the quintuple $(N, \searrow, \nearrow, \swarrow,\nwarrow)$  is a {commutative \nd algebra},
where
$$a\nwarro b= b\searro a:=a\lhd b\,\text{ and }\,a\nearro b =b\swarro a:= a\rhd b,\,{\forall a,b\in N.}$$
\end{remark}

As the commutative case of Proposition~\mref{prop:dendnd}, we have
\begin{prop}\label{prop:cdd}
A commutative differential dendriform algebra $(D,\prec,\succ, d)$ of weight zero
induces a pre-Novikov algebra $(D, \lhd, \rhd)$, where
\begin{eqnarray*}
a \lhd b := a \prec d(b)=d(b)\succ a,\quad a\rhd b :=d(b)\prec a=a\succ d(b), \quad\forall a,b\in {D}.
\end{eqnarray*}
Define
\begin{eqnarray*}
a\circ b:=a \lhd b+a\rhd b := a \prec d(b)+d(b)\prec a, \quad\forall a,b\in {D}.
\end{eqnarray*}
Then the pair $(D, \circ)$ is a Novikov algebra.
\end{prop}

\begin{proof}
We just check Eq.~\eqref{eq:pn1}, as the cases of Eqs.~\eqref{eq:pn2}-\eqref{eq:pn4} are similar. For any $a,b\in D$, we have
\begin{eqnarray*}
&&(a\lhd b)\lhd c-a\lhd (b\rhd c+b\lhd c)\\
&=&(a\prec d(b))\prec d(c)-a\prec d\big(d(c)\prec b+b\prec d(c)\big)\\
&=&(a\prec d(b))\prec d(c)-a\prec \big(d(d(c))\prec b+d(c)\prec d(b)+d(b)\prec d(c)+b\prec d(d(c))\big)\\
&=&-a\prec \big(d(d(c))\prec b+b\prec d(d(c))\big)\\
&=&-\big(a\prec d(d(c))\big) \prec b\\
&=&(d(a)\prec b)\prec d(c)-\big(a\prec d(d(c))+d(a)\prec d(c)\big) \prec b\\
&=&(d(a)\prec b)\prec d(c)-d\big(a\prec d(c)\big) \prec b\\
&=&(b\rhd a)\lhd c-b\rhd (a\lhd  c),
\end{eqnarray*}
as needed.
\end{proof}

The following is the tridendriform version of Definition~\mref{defn:nda}.
\begin{defn}\label{defn:nta}
A {\bf \ntd algebra} is a \bfk-module $T$ together with six binary operations $\searrow, \nearrow, \swarrow,\nwarrow, \vee$ and $\wedge$  on $T$
such that
\begin{eqnarray*}
(a \swarro b)\nwarro c&=&a\swarro (b\nearro c+b \nwarro c+b\wedg c), \label{eq:tnd1}\\
(a \searro b)\nwarro c&=&a\searro (b\nwarro c),\\
(a\searro b+a\swarro b+a \ve b)\nearro c&=&a\searro (b\nearro c),\label{eq:tnd3}\\
(a\nwarro b)\nwarro c-a\nwarro (b\searro c+b\swarro c+b\ve c)&=&a\swarro (b\searro c+b\swarro c+b\ve c)-(a \nwarro b)\swarro c,\label{eq:tnd4}\\
(a\nearro b)\nwarro c-a\nearro (b\swarro c)&=&a\searro (b\swarro c)-(a \nearro b)\swarro c,\\
(a\nearro b+a\nwarro b+a\wedg b)\nearro c-a\nearro (b\searro c)&=&a\searro (b\searro c)-(a\nearro b+a\nwarro b+a\wedg b)\searro c,\label{eq:tnd6}\\
(a\ve b)\nwarro c&=&a\ve (b\nwarro c),\\
(a\wedg b)\nwarro c-a\wedg (b\swarro c)&=&a\ve (b\swarro c)-(a\wedg b)\swarro c,\\
(a\swarro b)\wedg c&=&a\ve (b\nearro c),\\
(a\nwarro b)\wedg c-a\wedg (b\searro c)&=&a\ve (b\searro c)-(a\nwarro b)\ve c,\\
(a\searro b)\wedg c&=&a\searro (b\wedg c),\\
(a\nearro b)\wedg c-a\nearro (b\ve c)&=&a\searro (b\ve c)-(a\nearro b)\ve c,\\
(a \ve b)\wedg c&=&a\ve (b\wedg c),\\
(a\wedg b)\wedg c-a\wedg (b\ve c)&=&a\ve (b\ve c)-(a \wedg b)\ve c,\quad{\forall a,b,c\in T.}
\end{eqnarray*}
\end{defn}

\begin{prop} \label{prop:tdendnd}
A differential $q$-tridendriform algebra $(T,\prec,\succ, {\bullet}, d)$ of weight zero
induces a \ntd algebra $(T, \searrow, \nearrow, \swarrow,\nwarrow, \vee, \wedge)$,
where
\begin{align*}
a \searro b :=&~\searrod{a}{ b} , \quad a\nearro b :=\nearrod{a}{b}, \\
a \swarro b:=&~ \swarrod{a}{b}, \quad a\nwarro b:= \nwarrod{a}{b}, \\
a \ve b: =& ~ \ved{a}{b}, \quad a\wedge b:=\wedgd{a}{b},  \quad\forall a,b\in {T}.
\end{align*}
Moreover, if define
\begin{align*}
a\vdash b :=&\ a\searro b+a\swarro b+ a\ve b=d(a)\succ b+ d(a)\prec b+d(a){\bullet} b, \\
a\dashv b :=&\ a\nearro b+ a\nwarro b+a\wedg b= a\succ d(b)+ a\prec d(b)+a{\bullet} d(b),\quad\forall a,b\in T,
\end{align*}
then the triple $(T, \vdash,\dashv)$ is a Novikov-associative algerba.
\end{prop}
\begin{proof}
The proof is similar to the one of Proposition~\mref{prop:dendnd}, and we just check the first and last equations in Definition~\mref{defn:nta}. Let $a,b,c\in T$. Then
\begin{eqnarray*}
(a \swarro b)\nwarro c
&=&\nwarrod{(\swarrod{a}{  b})}{ c}\\
&=&\swarrod{a}{ (\nearrod{b}{ c}+\nwarrod{b}{  c}+\wedgd{b}{ c})}\\
&=&a\swarro (b\nearro c+b \nwarro c+b\wedg c),\\
(a\wedg b)\wedg c-a\wedg (b\ve c)
&=&\wedgd{(\wedgd{a}{ b})}{ c}-\wedgd{a}{ \ved{b}{ c}}\\
&=&\wedgd{(\wedgd{a}{ b})}{ c}-a{\bullet} \big(d(d(b)){\bullet} c+d(b){\bullet}d(c)\big)\\
&=&-a{\bullet} \big(d(d(b)){\bullet} c\big)\\
&=&-\big(a {\bullet} d(d(b))\big){\bullet} c\\
&=&\ved{a}{ (\ved{b}{ c})}-\big(d(a){\bullet} d(b)+a {\bullet} d(d(b))\big){\bullet} c\\
&=&\ved{a}{ (\ved{b}{ c})}-\ved{\wedgd{a}{  b}}{ c}\\
&=&a\ve (b\ve c)-(a \wedg b)\ve c.
\end{eqnarray*}
This completes the proof.
\end{proof}

The Manin black product of a binary quadratic operad $\cal P$ and the operad of post-Lie algebras is the post-$\cal P$ operad~\cite{BGN}. This provides a mechanism for the following concept by taking $\cal P$ to be the operad of Novikov algebras.

\begin{defn}\label{defn:pn}
 A {\bf post-Novikov algebra} is a \bfk-module $T$ with three binary operations $\lhd, \rhd$ and $ \vw$ on $T$ such that
 \begin{eqnarray*}
(a \rhd b)\lhd c&=& (a\rhd c+a \lhd c+a\vw c)\rhd b, \label{eq:tnd1}\\
(a \lhd b)\lhd c&=& (a\lhd c) \lhd b,\\
(a\lhd b)\lhd c-a\lhd (b\lhd c +b\rhd c+ b\vw c)&=&(b \rhd a) \lhd c-b \rhd(a \lhd c),\label{eq:tnd4}\\
(a\rhd b)\lhd c-a\rhd (c\rhd b)&=& (c\rhd b)\lhd a-c\rhd (a \rhd b),\\
(a\vw b)\lhd c&=&(a\lhd c)\vw b,\\
(a\vw b)\lhd c-a\vw (b\rhd c)&=&(b\rhd a)\vw c-b\rhd (a\vw c),\\
(a\rhd b)\vw c&=& (a\rhd c)\vw b,\\
(a\lhd b)\vw c-a\vw (b\lhd c)&=& (b\lhd a)\vw c-b\vw (a\lhd c),\\
(a \vw b)\vw c&=&(a\vw c)\vw b,\\
(a\vw b)\vw c-a\vw (b\vw c )&=& (b\vw a )\vw c-b\vw (a \vw c),\quad{\forall a,b,c\in T.}
\end{eqnarray*}
\end{defn}

\begin{remark}\label{re:pn}
If the 7-tuple $(T, \searrow, \nearrow, \swarrow,\nwarrow, \vee, \wedge)$ is a {commutative \ntd algebra}, then the quadruple
$(T, \nwarrow, \nearrow,\wedge)$ is a post-Novikov algebra.
Conversely, if the quadruple $(T, \lhd, \rhd, \vw)$ is a post-Novikov algebra, then the pair $(T, \vw)$ is a Novikov algebra and the 7-tuple $(T, \searrow, \nearrow, \swarrow,\nwarrow, \vee, \wedge)$ is a {commutative \ntd algebra}, where
$$a\nwarro b=b\searro a:= a\lhd b ,~ a\nearro b=b\swarro a:=a\rhd b  \,\text{ and }\, ~a\wedg b= b  \ve a:=a\vw b,\quad{\forall a,b\in T.} $$
\end{remark}

\begin{prop} \label{prop:ctdendnd}
A commutative differential $q$-tridendriform algebra $(T,\prec,\succ, {\bullet}, d)$ of weight zero
induces a post-Novikov algebra $(T, \lhd, \rhd, \vw)$,
where
\begin{align*}
a \lhd b :=&~ \nwarrod{a}{b}=\searrod{b}{ a} , \\
a \rhd b:=&~ \nearrod{a}{b}=\swarrod{b}{a}, \\
a \vw b: =& ~\wedgd{a}{b}= \ved{b}{a},  \quad\forall a,b\in{T}.
\end{align*}
Define
\begin{align*}
a\circ b:=a \lhd b+a \rhd b +a \vw b=\nwarrod{a}{b}+ \nearrod{a}{b}+\wedgd{a}{b},  \quad \forall a,b\in {T}.
\end{align*}
Then the pair $(T, \circ)$ is a Novikov algebra.
\end{prop}
\begin{proof}
It's just a commutative version of the proof of Proposition~\mref{prop:tdendnd}.
\end{proof}

\subsection{Koszul dual operads of operads of differential ($q$-tri)dendriform algebras of weight zero}
\nc\partia{\partial}
We are in a position to describe the Koszul dual operad of the operad of differential ($q$-tri)dendriform algebras of weight zero.

\begin{defn} \cite{Lod01}
A {\bf diassociative algebra} is a \bfk-module $D$ with two binary operations
$\vdash$ and  $\dashv$ on $D$ such that
\begin{align*}
(a\dashv b)\dashv c&=~a\dashv(b\dashv c),\\
(a\dashv b)\dashv c&=~a\dashv(b\vdash c),\\
(a\vdash b)\dashv c&=~a\vdash (b\dashv c),\\
(a \dashv b)\vdash c&=~a\vdash (b\vdash c),\\
(a\vdash b)\vdash c&=~a\vdash (b\vdash c),\quad{\forall a,b,c\in D.}
\end{align*}
\end{defn}

\begin{remark}
Let $A$ be an algebra and ${\rm End}(A)$ be the set of all linear maps from $A$ to $A$. The
{\bf centroid} of $A$~\cite{FRH} is defined to be
$${\rm Cent}(A) := \{\partia\in {\rm End}(A) \,| \,\partia (ab) = \partia (a)b = a\partia(b),\quad\forall a,b\in A\}.$$
\end{remark}

\begin{prop}
The Koszul dual operad of the operad of differential dendriform algebras of weight zero is the operad of diassociative algebras with a centroid consisting of one  square-zero linear operator.
\mlabel{pp:dualdiffdend}
\end{prop}
\begin{proof}
It follows from the fact that the Koszul dual operad of the operad of dendriform algebras is the operad of diassociative algebras{~\mcite{Lod01}} and the fact that the Koszul duality of a derivation is a square-zero linear operator in the centroid~\cite[Proposition 7.2]{Lod10}.
\end{proof}

The following concept is a generalization of the triassociative algebra~\cite{LoRo04}.

\begin{defn}
Let $q\in\bfk$. A {\bf $q$-triassociative algebra} is a \bfk-module $T$ with three binary operations
$\vdash$ , $\dashv$ and $\perp$ on $T$ such that
\begin{equation*}
\begin{split}
(a\dashv b)\dashv c&=~a\dashv(b\dashv c),\\
(a\dashv b)\dashv c&=~a\dashv(b\vdash c),\\
(a\vdash b)\dashv c&=~a\vdash (b\dashv c),\\
(a \dashv b)\vdash c&=~a\vdash (b\vdash c),\\
(a\vdash b)\vdash c&=~a\vdash (b\vdash c),\\
(a\perp b)\perp c&=~a\perp (b\perp c),\quad{\forall a,b,c\in T.}
\end{split}
\quad
\begin{split}
(a\dashv b)\dashv c&=~qa\dashv (b\perp c),\\
(a\perp b)\dashv c&=~a\perp (b\dashv c),\\
(a\dashv b)\perp c&=~a\perp (b\vdash c),\\
(a\vdash b)\perp c&=~a\vdash (b\perp c),\\
q(a\perp b)\vdash c&=~a\vdash (b\vdash c),\\
\,
\end{split}
\end{equation*}
\end{defn}

\begin{prop}
Let $q\in \bfk$ be invertible.
The Koszul dual operad of the operad of differential $q$-tridendriform algebras of weight zero is the operad of $\frac{1}{q}$-triassociative algebras with a centroid consisting of one  square-zero linear operator.
\mlabel{pp:dualdiffdend}
\end{prop}

\begin{proof}
Let
$$V:=\bfk \{\prec,\succ,\bullet\}\,\text{ and }\, V^\ast :=\bfk\{\dashv,\vdash,\perp\}:=\bfk\{\prec^\ast,\succ^\ast,\bullet^\ast\}.$$
Define
$$ \langle ~, ~\rangle: \bfk \Big\{(a\cdot_\mu b)\cdot_\nu c,\, a\cdot_\mu (b\cdot_\nu c)\mid \cdot_\mu,\cdot_\nu\in V \Big\} \otimes
\bfk \Big\{(a\cdot_\mu^\ast b)\cdot_\nu^\ast c, \, a\cdot_\mu^\ast (b\cdot_\nu^\ast c)\mid \cdot_\mu,\cdot_\nu\in V\Big\}\to \bfk$$
by setting
\begin{equation*}
\begin{split}
&\langle (a\cdot_\mu b)\cdot_\nu c,\, (a\cdot_\mu^\ast b)\cdot_\nu^\ast c\rangle :=1,\\
&\langle a\cdot_\mu (b\cdot_\nu c),\, a\cdot_\mu^\ast(b \cdot_\nu^\ast c)\rangle :=-1, \\
&\langle ~, ~\rangle :=0,\text{ otherwise}.
\end{split}
\end{equation*}
Denote by
 \begin{align*}
R_{\rm trid}:=&\ \Big\{
(a \prec b) \prec c- \ a \prec( b \star_q c),
 (a \succ b) \prec c- \ a \succ(b \prec c),
 (a\star_q b) \succ c-\ a \succ(b \succ c),\\
&\ \ (a \succ b) \bullet c- \ a \succ(b \bullet c),
(a \prec b) \bullet c- \ a \bullet(b \succ c),
 (a \bullet b) \prec c- \ a \bullet(b \prec c),\\
&\ \ (a \bullet b) \bullet c- \ a \bullet(b \bullet c)\Big\},
\end{align*}
the set of relations of $q$-tridendriform algebras in Definition~\ref{defn:qtrid}.
Then the set $R_{\rm tria}$ of relations of the Koszul dual operad  of $q$-tridendriform algebras is given by
$$R_{\rm tria}=\left\{ r\in\bfk\{(a\cdot_\mu^\ast b)\cdot_\nu^\ast c, \, a\cdot_\mu^\ast (b\cdot_\nu^\ast c) \mid \cdot_\mu,\cdot_\nu\in V\} \,\,\big|\,\,
\langle R_{\rm trid}, r \rangle =0\right\}.$$
By a direct computation,
$R_{\rm tria}$ is the set of relations of $\frac{1}{q}$-triassociative algebras.
Further in light of~\cite[Proposition 7.2]{Lod10}, the Koszul duality of a derivation is a square-zero linear operator in the centroid.
\end{proof}

\section{Free weighted commutative differential ($q$-tri)dendriform algebras}\mlabel{sec:difftridend}
In this section, we mainly construct free commutative differential $q$-tridendriform algebras of weight $\lambda$ via the quasi-shuffle product of weight $q$. As an application, free differential dendriform algebras of weight $\lambda$ are obtained via the shuffle product.

\subsection{Free commutative (tri)dendriform algebras}
This subsection is devoted to recall the construction of free tridendriform (resp. dendriform) algebras via the quasi-shuffle (resp. shuffle) product.

For a \bfk-module $V$, denote by
\[
T(V):=\bigoplus_{k \geqslant 0} V^{\otimes k}\,\text{ and }\,  T^{+}(V) :=\bigoplus_{k \geqslant 1} V^{\otimes k}.
\]
Let $A$ be a $\mathbf{k}$-algebra and $q\in \bfk$.
For
\begin{equation}\label{eq:frakab}
\mathfrak{a} :=a_1 \otimes \mathfrak{a}^{\prime}:=a_1 \otimes \cdots \otimes a_m  \in A^{\otimes m}, \quad \mathfrak{b}:=b_1 \otimes \mathfrak{b}^{\prime}:= b_1 \otimes \cdots \otimes b_n  \in A^{\otimes n},
\end{equation}
the quasi-shuffle product $\ast_q$ of weight $q$ on $T(A)$ is defined by~\cite{Guo12}
\begin{equation*}
\mathfrak{a} \ast_q \mathfrak{b} :=a_1 \otimes\big(\mathfrak{a}^{\prime} \ast_q \mathfrak{b}\big)+b_1 \otimes\big(\mathfrak{a} \ast_q \mathfrak{b}^{\prime}\big)+q\big(a_1 b_1\big) \otimes\big(\mathfrak{a}^{\prime} \ast_q \mathfrak{b}^{\prime}\big).
\end{equation*}
This product is restricted to an associative product $\ast_q$ on $T^{+}(A)$.
When $q =0$, $\ast_0$ is the shuffle product. When $q =1$, $\ast_1$ is the quasi-shuffle product~\cite{Hoffman}.
With the above notations, we have the following result.

\begin{lemma}~\cite[Theorem 5.2.4]{Guo12}\label{thm:tridend}
{\rm{\begin{enumerate}
\item \label{item:dend}
Let $V$ be a \bfk-module. In the shuffle product algebra $\left(T^{+}(V), \ast_0 \right)$, define binary operations $\prec_V$ and $\succ_V$ by
\[
\mathfrak{a} \prec_V \mathfrak{b}:=a_1 \otimes\left(\mathfrak{a}^{\prime}\ast_0\mathfrak{b}\right), \, \mathfrak{a} \succ_V \mathfrak{b}:=b_1 \otimes\left(\mathfrak{a}\ast_0\mathfrak{b}^{\prime}\right),
\]
for pure tensors $\mathfrak{a}=a_1 \otimes \mathfrak{a}^{\prime}, \mathfrak{b}=b_1 \otimes \mathfrak{b}^{\prime}$. Then the triple $\left(T^{+}(V), \prec_V, \succ_V\right)$, together with the natural embedding $V\hookrightarrow T^{+}(V)$, is the free commutative dendriform algebra on $V$.
\item \label{item:tridend} Let $A$ be a $\mathbf{k}$-algebra. In the quasi-shuffle product algebra $\left(T^{+}(A), \ast_1\right)$, define binary operations $\prec_A, \succ_A, \bullet_A$ by
$$
\begin{aligned}
\mathfrak{a} \prec_A \mathfrak{b}:=& a_1 \otimes\left(\mathfrak{a}^{\prime} \ast_1 \mathfrak{b}\right), \,
\mathfrak{a} \succ_A \mathfrak{b}:=b_1 \otimes\left(\mathfrak{a} \ast_1 \mathfrak{b}^{\prime}\right), \,
\mathfrak{a} \bullet_A \mathfrak{b}:= \left(a_1 b_1\right) \otimes\left(\mathfrak{a}^{\prime} \ast_1 \mathfrak{b}^{\prime}\right),
\end{aligned}
$$
for pure tensors $\mathfrak{a}=a_1 \otimes \mathfrak{a}^{\prime}$ and $\mathfrak{b}=b_1 \otimes \mathfrak{b}^{\prime}$. Then the quadruple $\left(T^{+}(A), \prec_A, \succ_A,
\bullet_A\right)$, together with the natural embedding $A \hookrightarrow T^{+}(A)$, is the free commutative tridendriform algebra on $A$.
\end{enumerate}}}
\end{lemma}

\begin{remark}
In Lemma~\ref{thm:tridend}-\ref{item:tridend}, let $\ast_q$ replace $\ast_1$, we obtain that the quadruple $\left(T^{+}(A), \prec_A, \succ_A,
\bullet_A\right)$, together with the natural embedding $A \hookrightarrow T^{+}(A)$, is the free commutative $q$-tridendriform algebra on $A$.
\mlabel{rk:free comm qtri}
\end{remark}

\subsection{Free weighted commutative differential $q$-tridendriform algebras}
This subsection focuses on the construction of free commutative differential $q$-tridendriform algebras of weight $\lambda$.

Let $(A, d_0)$ be a commutative differential \bfk-algebra of weight $\lambda$.
For $\frak a=a_1\ot\frak a'\in A^{\otimes m}$ and $\frak b=b_1\ot\frak b'\in A^{\otimes n}$ as in Eq.~\eqref{eq:frakab},
define binary operations $\prec_A, \succ_A, \bullet_A:A \otimes A \rightarrow A$ by
\begin{equation*}
\mathfrak{a} \prec_A \mathfrak{b}:= a_1 \otimes\big(\mathfrak{a}^{\prime} \ast_q \mathfrak{b}\big),\,
\mathfrak{a} \succ_A
\mathfrak{b}:= b_1 \otimes\big(\mathfrak{a} \ast_q\mathfrak{b}^{\prime}\big),\,
\mathfrak{a} \bullet_A \mathfrak{b}:= \big(a_1 b_1\big) \otimes\big(\mathfrak{a}^{\prime} \ast_q \mathfrak{b}^{\prime}\big).
\end{equation*}
By Remark~\mref{rk:free comm qtri}, the quadruple  $\left(T^{+}(A), \prec_A, \succ_A,
\bullet_A\right)$ is the free commutative $q$-tridendriform algebra on $A$.
So for each pure tensor $\fraka = a_1\otimes \cdots \otimes a_m$ in $A^{\otimes m}$, there is a unique $m$-arity operation $f_{\frak a}$ in the operad of commutative $q$-tridendriform algebras such that $\fraka = a_1\otimes \cdots \otimes a_m =f_{\frak a}(a_1, \cdots, a_m)$.

By Proposition~\ref{prop:wdiff}, the linear operator $d_0: A\rightarrow A$
extends to a unique derivation $d_A: T^{+}(A) \rightarrow T^+(A)$. In more details, for $m\geq 1$ and
\begin{align*}
\fraka = a_1\ot\cdots\ot a_m = f_\fraka(a_1, \cdots, a_m) \in A^{\otimes m},
\end{align*}
we have
\begin{align}\label{eq:diffindu}
d_A(\frak a)=& \ d_A(f_{\frak a}(a_1,\cdots,a_m))\nonumber\\
=& \ \sum_{k=1}^{m} \lambda^{k-1}\bigg(\sum_{1\leq i_1<\cdots< i_k \leq m}
f_\fraka(a_1,\cdots, d_0(a_{i_1}),\cdots,d_0(a_{i_2}),\cdots,d_0(a_{i_k}),\cdots,a_m\bigg).
\end{align}
For example, if $\frak a=a_1$, then
\begin{equation}
d_A(a_1) = d_0(a_1).
\mlabel{eq:da01}
\end{equation}
If $\frak a=a_1\ot a_2=a_1\prec_A a_2=:f_{\frak a}(a_1, a_2)$, then
\begin{align*}
d_A(\frak a)=& \ d_A(f_{\frak a}(a_1,a_2))\\
=& \ f_{\frak a}(a_1,d_0(a_2))+f_{\frak a}(d_0(a_1), a_2)+\lambda f_{\frak a}(d_0(a_1),d_0(a_2))\quad\text{(by Eq.~\eqref{eq:diffindu})}\\
=& \ a_1\prec_A d_0(a_2)+d_0(a_1)\prec_A a_2+\lambda d_0(a_1)\prec d_0(a_2).
\end{align*}

\begin{prop}
Let $(A, d_0)$ be a commutative differential \bfk-algebra of weight $\lambda$.
Then the quintuple $(T^{+}(A), \prec_A,\succ_A, \bullet_A, d_A)$
is a commutative differential $q$-tridendriform algebra of weight $\lambda$.
\mlabel{prop:difftridend}
\end{prop}

\begin{proof}
The proof follows from Remark~\ref{rk:free comm qtri} and the fact that
the $d_A$ is a derivation on $T^{+}(A)$ of weight $\lambda$.
\end{proof}

Now we are ready for the main result in this subsection.

\begin{theorem}
Let $(A, d_0)$ be a commutative differential \bfk-algebra of weight $\lambda$.
Then the quintuple $\left(T^{+}(A), \prec_A,\succ_A, \bullet_A, d_A\right)$, together with natural embedding
$j_A: A\hookrightarrow T^{+}(A)$,
is the free commutative differential $q$-tridendriform algebra of weight $\lambda$ on $(A, d_0)$.
\mlabel{thm:difftridend}
\end{theorem}

\begin{proof}
By Proposition~\ref{prop:difftridend}, the quintuple $\left(T^{+}(A), \prec_A,\succ_A, \bullet_A, d_A\right)$ is a commutative differential $q$-tridendriform algebra of weight $\lambda$. To prove the universal property, let the quintuple $(T,  \prec_T, \succ_T,\bullet_T, d_T)$ be a commutative differential $q$-tridendriform algebra of weight $\lambda$ and let $\psi:\left(A, \mathrm{d}_0\right) \rightarrow(T, \bullet_T, d_T )$ be a differential algebra morphism. Then $\psi$ is an algebra homomorphism and
\begin{equation}
d_T\psi = \psi d_0. \mlabel{eq:dted0}
\end{equation} So by Remark~\ref{rk:free comm qtri}, there is a unique $q$-tridendriform algebra homomorphism \[\lbar{\psi}:\left(T^+(A), \prec_A, \succ_A, \bullet_A\right)\rightarrow(T, \prec_T, \succ_T, \bullet_T)\]
 such that
$
\psi=\lbar{\psi} j_A.
$
\[\xymatrix@C=2cm{
 A \ar[d]_{\psi} \ar[r]^-{j_A}
                & T^{+}(A)
                \ar@{.>}[dl]_{\lbar{\psi}} \\
 T}
\]
It remains to prove
\begin{equation*}
d_T\Big(\lbar{\psi}(a_1\ot\cdots\ot a_m)\Big)=\lbar{\psi}\Big(d_A(a_1\ot\cdots\ot a_m)\Big)
\end{equation*}
for all pure tensors $a_1\ot\frak a':= a_1\ot\cdots\ot a_m\in T^{+}(A)$. We employ induction on $m\geq 1$. For the initial step of $m=1$, it follows from Eqs.~(\mref{eq:da01}) and~(\mref{eq:dted0}) that
\begin{align*}
d_T\Big(\lbar{\psi}(a_1)\Big)=& \ d_T\Big(\lbar{\psi}(j_A(a_1))\Big)
=d_T(\psi(a_1))=\psi(d_0(a_1))\\
=& \ \lbar{\psi}\Big(j_A(d_0(a_1))\Big)= \lbar{\psi}\Big(d_0(a_1)\Big)= \lbar{\psi}\Big(d_A(a_1)\Big).
\end{align*}
For the inductive step of $m\geq 2$,
we get
\begin{align*}
\lbar{\psi}\Big(d_A(a_1\ot\cdots\ot a_m)\Big)
=& \ \lbar{\psi}\Big(d_A(a_1\ot\frak a')\Big)
= \lbar{\psi}\Big(d_A(a_1\prec_A\frak a')\Big)\\
=& \ \lbar{\psi}\Big(d_0(a_1)\prec_A\frak a'+a_1\prec_A d_A(\frak a')+\lambda d_0(a_1)\prec_A d_A(\frak a')\Big)\quad\text{(by Eqs.~\eqref{eq:difftri2} and~(\mref{eq:da01}))}  \\
=& \ \lbar{\psi}(d_0(a_1))\prec_T\lbar{\psi}(\frak a')
+\lbar{\psi}(a_1)\prec_T\lbar{\psi}(d_A(\frak a'))+\lambda\lbar{\psi}(d_0(a_1))\prec_T \lbar{\psi}(d_A(\frak a'))\\
&\hspace{3cm}\text{(by $\lbar{\psi}$ being a $q$-tridendriform algebra homomorphism)}\\
=& \ \lbar{\psi}  j_A (d_0(a_1))\prec_T\lbar{\psi}(\frak a')
+\lbar{\psi}(a_1)\prec_T\lbar{\psi}(d_A(\frak a'))+\lambda\lbar{\psi}  j_A(d_0(a_1))\prec_T \lbar{\psi}(d_A(\frak a'))\\
=& \ \psi (d_0(a_1))\prec_T\lbar{\psi}(\frak a')
+\lbar{\psi}(a_1)\prec_T\lbar{\psi}(d_A(\frak a'))+\lambda\psi(d_0(a_1))\prec_T \lbar{\psi}(d_A(\frak a'))\\
=& \ d_T (\psi(a_1))\prec_T\lbar{\psi}(\frak a')
+\lbar{\psi}(a_1)\prec_T\lbar{\psi}(d_A(\frak a'))+\lambda d_T(\psi(a_1))\prec_T \lbar{\psi}(d_A(\frak a'))\\
&\hspace{7cm}\text{(by Eq.~(\mref{eq:dted0}))}\\
=& \ d_T (\lbar{\psi}  j_A(a_1))\prec_T\lbar{\psi}(\frak a')
+\lbar{\psi}(a_1)\prec_T\lbar{\psi}(d_A(\frak a'))+\lambda d_T(\lbar{\psi}  j_A(a_1))\prec_T \lbar{\psi}(d_A(\frak a'))\\
=& \ d_T (\lbar{\psi}(a_1))\prec_T\lbar{\psi}(\frak a')
+\lbar{\psi}(a_1)\prec_T\lbar{\psi}(d_A(\frak a'))+\lambda d_T(\lbar{\psi}(a_1))\prec_T \lbar{\psi}(d_A(\frak a'))\\
=& \ d_T (\lbar{\psi}(a_1) )\prec_T\lbar{\psi}(\frak a')
+\lbar{\psi}(a_1)\prec_T d_T (\lbar{\psi}(\frak a') )+\lambda d_T (\lbar{\psi}(a_1) )\prec_T d_T (\lbar{\psi}(\frak a') )\\
&\hspace{3cm}\text{(by the induction hypothesis)}\\
=\ &d_T\Big(\lbar{\psi}(a_1)\prec_T\lbar{\psi}(\frak a')\Big)= d_T\Big(\lbar{\psi}(a_1\prec_A\frak a')\Big)= d_T\Big(\lbar{\psi}(a_1\ot\frak a')\Big)\\
=& \ d_T\Big(\lbar{\psi}(a_1\ot\cdots\ot a_m)\Big).
\end{align*}
 This completes the proof.
\end{proof}

As an application, we construct the free weighted commutative differential dendriform algebra.

\begin{coro}
Let $V$ be a \bfk-module. Then
the quadruple $\left(T^{+}(V), \prec_V,\succ_V, d_V\right)$, together with the natural embedding $V\hookrightarrow T^{+}(V)$,
is the free commutative differential dendriform algebra of weight $\lambda$ on $V$.
\mlabel{coro:freecddend}
\end{coro}

\begin{proof}
It follows from Theorem~\ref{thm:difftridend} by taking $q=0$ and $a\bullet b :=0$ for $a, b\in V$.
\end{proof}

\section{Free weighted differential $q$-tridendriform and dendriform algebras}
\mlabel{sec:freedifftridend}
In this section, we construct free differential $q$-tridendriform algebras of weight $\lambda$ in terms of valently decorated Schr\"oder trees.
For this, let us first recall the construction of free $q$-tridendriform algebras. See~\cite{BR, Chapoton,LoRo98,LoRo04} for more details.

\subsection{Free weighted differential $q$-tridendriform algebras}
\subsubsection{Free $q$-tridendriform algebras}\mlabel{sub:trid}
 Let $X$ be a set. For $n\geq 0,$
let $T_{n,\,X}$ be the set of planar rooted trees with $n+1$ leaves
and with vertices valently decorated by elements of $X$, in the sense that if a vertex has valence $k$, then the vertex is decorated by an element in $X^{k-1}$. For example, the vertex of $\stree x$ is decorated by $x\in X$ while the vertex of
$\XX{\xx002 \xxh0023{x\ }{0.5} \xxh0012{\ \,y}{0.4}}$ is decorated by $(x,y)\in X^{2}$.
Here are the first few of them:
\begin{align*}
T_{0,\,X}=& \ \{|\},\qquad T_{1,\,X}=\left\{\stree x\Bigm|x\in X\right\},
\qquad T_{2,\,X}=\left\{
\XX{\xxr{-5}5
\xxhu00x \xxhu{-5}5y
}, \,
\XX{\xxl55
\xxhu00x \xxhu55y
}, \,
\XX{\xx002
\xxh0023{x\ }{0.5} \xxh0012{\ \,y}{0.4}
}\Bigm|x,y\in X
\right\},\\
T_{3,\,X}=& \ \left\{
\XX[scale=1.6]{\xxr{-4}4\xxr{-7.5}{7.5}
\xxhu00x \xxhu[0.1]{-4}4{\,y} \xxhu[0.1]{-7.5}{7.5}{z}
},
\XX[scale=1.6]{\xxr{-5}5\xxl{-2}8
\xxhu00x
\xxhu[0.1]{-5}5{y} \xxhu[0.1]{-2}8{z}
},
\XX[scale=1.6]{\xxr{-4}4\xx{-4}42
\xxhu00x \xxh{-4}423{y\ \,}{0.3} \xxh{-4}412{\ \, z}{0.3}
},
\XX[scale=1.6]{\xx00{1.6}\xx00{2.4}
\xxh001{1.6}{\ \ \,z}{0.6}
\xxh00{1.6}{2.4}{y}{0.5}
\xxh00{2.4}3{x\ \ }{0.6}
},
\XX[scale=1.6]{\xxr{-6}6\xxl66
\xxhu00{x} \xxhu66{z} \xxhu[0.1]{-6}6{y}
},
\XX[scale=1.6]{\xxlr0{7.5} \draw(0,0)--(0,0.75);
\xxh0023{x\ \,}{0.3} \xxhu[0.12]0{7.5}{z}
\xxh0012{\ \ y}{0.22}
},\ldots\Bigg|\,x,y,z\in X
\right\}.
\end{align*}
Elements in $\bigcup\limits_{n\geq 0} T_{n, X}$ are called {\bf valently decorated Schr\"oder trees} by $X$.

\begin{remark}
In the graphical representation above, the edge pointing downwards is the root, the upper edges are the leaves. The other edges, joining two internal vertices, are called internal edges.
\end{remark}
\nc\ta{{\tau}}
\nc\sigm{{\sigma}}
For $\ta^{(i)}\in T_{n_i,\,X}$ with $0\leq i\leq m$ and $x_1,\ldots,x_m\in X,$ the grafting $\bigvee$ of $\ta^{(i)}$ over $(x_1,\ldots,x_m)$ is
\begin{equation}\mlabel{eq:texpr}
\ta = \bigvee\nolimits_{x_1,\ldots,x_m}^{m+1}(\ta^{(0)},\cdots,\ta^{(m)}),
\end{equation}
obtained by joining $m+1$ roots of $\ta^{(i)}$ to a new root valently decorated by
$(x_1, \ldots, x_m)$.
Conversely, any valently decorated Schr\"oder tree $\ta$ can be uniquely expressed as such a grafting of lower depth valently decorated Schr\"oder trees in Eq.~(\mref{eq:texpr}). The {\bf depth} $\dep{(\ta)}$ of a rooted tree $\ta$ is the maximal length of linear chains of vertices from the root to the leaves of the tree.
For example,
\[\dep\biggl(\stree x\biggr) = 1\, \text{ and } \, \dep\biggl(\XX{\xxr{-5}5
\xxhu00x \xxhu{-5}5y
}\biggr) = 2.\]
The {\bf breadth} $\bre(\ta)$ of $\ta$ in Eq.~(\mref{eq:texpr}) is defined to be $\bre(\ta):=m+1$. For example,
\begin{align*}\XX{\xx002
\xxh0023{x\ }{0.5} \xxh0012{\ \,y}{0.4}
}=&\, \bigvv x,y;3;(|,|,|),\quad  \bre\biggl(\XX{\xx002
\xxh0023{x\ }{0.5} \xxh0012{\ \,y}{0.4}
}\biggr) = 3\, \text{ and }\\
\ta =& \XX[scale=1.6]{\xxlr0{7.5}\xxl77 \draw(0,0)--(0,0.75);
\xxh0023{x\ \,}{0.3} \xxhu[0.1]0{7.5}{y} \xxhu77{u}
\node at (0.15,0.38) {$z$};
}
=\bigvv x,z;3;\biggl(|,\stree y,\stree u\biggr), \quad \bre(\ta) = 3.
\end{align*}

Let $\mathrm{DT}(X):=\underset{n\geq 1}\bigoplus\,\bfk T_{n,\,X}.$ Define binary operations $\prec,\succ$ and $\bullet$ on $\mathrm{DT}(X)$ recursively on $\dep(\ta)+\dep(\sigm)$ as follows.
\begin{enumerate}
\item For $\ta\in T_{n,\,X}$ with $n\geq 1$, put
$$|\succ \ta:=\ta\prec |:=\ta,\, |\prec \ta:=\ta\succ |:=0 \text{ and }\, |\bullet \ta:=\ta\bullet |:=0.$$

\item For $\ta=\bigvee_{x_1,\ldots,x_m}^{m+1}(\ta^{(0)},\cdots,\ta^{(m)})$ and $\sigm=\bigvee_{y_1,\ldots,y_n}^{n+1}(\sigm^{(0)},\cdots,\sigm^{(n)}),$   set
{\small{\begin{align}
 \ta\prec \sigm:=& \ \bigvee\nolimits_{x_1,\ldots,x_m}^{m+1}(\ta^{(0)},\cdots,\ta^{(m-1)},\ta^{(m)}\succ \sigm+\ta^{(m)}\prec \sigm+q \ta^{(m)}\bullet \sigm),\nonumber \\
 \ta\succ \sigm:=& \ \bigvee\nolimits_{y_1,\ldots,y_n}^{n+1}(\ta\succ \sigm^{(0)}+\ta\prec \sigm^{(0)}+q \ta\bullet \sigm^{(0)},\sigm^{(1)},\cdots,\sigm^{(n)}),\nonumber\\
 \ta\bullet \sigm:=& \ \bigvee\nolimits_{x_1,\ldots,x_m,\,y_1,\ldots,y_n}^{m+n+1}(
\ta^{(0)},\cdots,\ta^{(m-1)},\ta^{(m)}\succ \sigm^{(0)}+\ta^{(m)}\prec \sigm^{(0)}+q \ta^{(m)}\bullet \sigm^{(0)},\sigm^{(1)},\cdots,\sigm^{(n)}). \nonumber
\end{align}}}
\end{enumerate}
Here we employ the convention that
$|\prec |+|\succ |+q\, |\bullet |=|$ provided $\ta^{(m)}=|=\sigm^{(0)}$.

Denote $j_X:X\ra\mathrm{DT}(X),\, x\mapsto \stree x$.

\begin{lemma}\cite{BR}\label{thm:freeqtri}
Let $X$ be a set and $q\in \bfk$. Then the quadruple $(\mathrm{DT}(X),\prec,\succ,\bullet)$, together with the map $j_X$, is the free $q$-tridendriform algebra on $X$.
\mlabel{thm:freetridendriform}
\end{lemma}

\subsubsection{Free weighted differential $q$-tridendriform algebras}
In this subsection, we aim to construct the free differential $q$-tridendriform algebra of weight $\lambda$ in terms of valently decorated Schr\"oder trees.

Let $X$ be a set. Denote
$\Delta X :=X \times \mathbb{N}=\left\{x^{(n)} \mid x \in X, n \geq 0\right\}$.
By Lemma~\ref{thm:freetridendriform}, the quadruple $(\mathrm{DT}(\Delta X),\prec,\succ,\bullet)$ is the free $q$-tridendriform algebra on $\Delta X$,  with $\prec,\succ,\bullet$ given recursively on $\dep(\ta)+\dep(\sigm)$ as follows.
\begin{enumerate}
\item For $\ta\in T_{n,\, \Delta X}$ with $n\geq 1$, define
$$|\succ \ta:=\ta\prec |:=\ta,\, |\prec \ta:=\ta\succ |:=0 \text{ and }\, |\bullet \ta:=\ta\bullet |:=0.$$
\item For $\ta=\bigvee_{x_1^{(r_1)},\ldots,x_m^{(r_m)}}^{m+1}(\ta^{(0)},\cdots,\ta^{(m)})$ and $\sigm=\bigvee_{y_1^{(s_1)},\ldots,y_n^{(s_n)}}^{n+1}(\sigm^{(0)},\cdots,\sigm^{(n)}),$ set
{\small{\begin{align}
 \ta\prec \sigm:=& \ \bigvee\nolimits_{x_1^{(r_1)},\ldots,x_m^{(r_m)}}^{m+1}(\ta^{(0)},\cdots,\ta^{(m-1)},\ta^{(m)}\succ \sigm+\ta^{(m)}\prec \sigm+q \ta^{(m)}\bullet \sigm), \nonumber\\
 \ta\succ \sigm:=& \ \bigvee\nolimits_{y_1^{(s_1)},\ldots,y_n^{(s_n)}}^{n+1}(\ta\succ \sigm^{(0)}+\ta\prec \sigm^{(0)}+q \ta\bullet \sigm^{(0)},\sigm^{(1)},\cdots,\sigm^{(n)}), \nonumber\\
 %
 \ta\bullet \sigm:=& \ \bigvee\nolimits_{x_1^{(r_1)},\ldots,x_m^{(r_m)},\,y_1^{(s_1)},\ldots,y_n^{(s_n)}}^{m+n+1}(
\ta^{(0)},\cdots,\ta^{(m-1)},\ta^{(m)}\succ \sigm^{(0)}+\ta^{(m)}\prec \sigm^{(0)}+q \ta^{(m)}\bullet \sigm^{(0)},\sigm^{(1)},\cdots,\sigm^{(n)}).\nonumber
\end{align}}}
\end{enumerate}

By the freeness of $\mathrm{DT}(\Delta X)$ as a $q$-tridendriform algebra on $\Delta X$, each valently decorated Schr\"{o}der tree $\ta\in \mathrm{DT}(\Delta X)$ can be uniquely expressed as
\begin{equation*}
  \ta=f_\ta(x_1^{(r_1)},\cdots, x_n^{(r_n)})
\end{equation*}
for some $x_1^{(r_1)},\cdots, x_n^{(r_n)}\in\Delta X$ and some $n$-arity operation $f_\ta$ in the operad of $q$-tridendriform algebras.

In terms of Proposition~\mref{prop:wdiff}, the map
$$\Delta X \rightarrow \mathrm{DT}(\Delta X), \, x^{(n)}\mapsto x^{(n+1)}$$
extends to a unique derivation
$d_X: \mathrm{DT}(\Delta X) \rightarrow \mathrm{DT}(\Delta X)$
by
\begin{eqnarray}\label{eq:defndiff}
d_X(\ta)\nonumber&=&d_X\big(f_\ta(x_1^{(r_1)},\cdots, x_n^{(r_n)})\big)\nonumber\\
&=&
\sum_{k=1}^{n} \lambda^{k-1}\bigg(\sum_{1\leq i_1<\cdots< i_k \leq n}
f_\ta(x_1^{(r_1)},\cdots, x_{i_1}^{(r_{i_1}+1)},\cdots,x_{i_2}^{(r_{i_2}+1)},\cdots,x_{i_k}^{(r_{i_k}+1)},
\cdots,x_n^{(r_n)})\bigg).
\end{eqnarray}
For example, for
$$\ta=\bigvv {x_2^{(r_2)}};2;(\Dtree{x_1^{(r_1)}},\Dtree{x_3^{(r_3)}}),$$
we have
$$\ta=\biggl(\Dtree{x_1^{(r_1)}}\succ\Dtree{x_2^{(r_2)}}\biggr)\prec \Dtree{x_3^{(r_3)}}=:f_\ta({x_1^{(r_1)}},{x_2^{(r_2)}},{x_3^{(r_3)}}),$$
and so
\begin{eqnarray*}
d_X(\ta)&=&d_X\big(f_\ta({x_1^{(r_1)}},{x_2^{(r_2)}},{x_3^{(r_3)}})\big)\\
&=&
f_\ta({x_1^{(r_1+1)}},{x_2^{(r_2)}},{x_3^{(r_3)}})+f_\ta({x_1^{(r_1)}},{x_2^{(r_2+1)}},{x_3^{(r_3)}})+f_\ta({x_1^{(r_1)}},{x_2^{(r_2)}},{x_3^{(r_3+1)}})\\
&&+\lambda f_\ta({x_1^{(r_1+1)}},{x_2^{(r_2+1)}},{x_3^{(r_3)}})+\lambda f_\ta({x_1^{(r_1+1)}},{x_2^{(r_2)}},{x_3^{(r_3+1)}})+\lambda f_\ta({x_1^{(r_1)}},{x_2^{(r_2+1)}},{x_3^{(r_3+1)}})\\
&&+\lambda^2 f_\ta({x_1^{(r_1+1)}},{x_2^{(r_2+1)}},{x_3^{(r_3+1)}})\\
&=& \bigvv {x_2^{(r_2)}};2;(\Dtree{x_1^{(r_1+1)}},\Dtree{x_3^{(r_3)}})+\bigvv {x_2^{(r_2+1)}};2;(\Dtree{x_1^{(r_1)}},\Dtree{x_3^{(r_3)}})+\bigvv {x_2^{(r_2)}};2;(\Dtree{x_1^{(r_1)}},\Dtree{x_3^{(r_3+1)}})\\
&&+\lambda\bigvv {x_2^{(r_2+1)}};2;(\Dtree{x_1^{(r_1+1)}},\Dtree{x_3^{(r_3)}})
+\lambda \bigvv {x_2^{(r_2)}};2;(\Dtree{x_1^{(r_1+1)}},\Dtree{x_3^{(r_3+1)}})
+\lambda \bigvv {x_2^{(r_2+1)}};2;(\Dtree{x_1^{(r_1)}},\Dtree{x_3^{(r_3+1)}})\\
&&+\lambda^2\bigvv {x_2^{(r_2+1)}};2;(\Dtree{x_1^{(r_1+1)}},\Dtree{x_3^{(r_3+1)}}).
\end{eqnarray*}

We state the following result as a preparation.

\begin{prop}
Let $X$ be a set and $\lambda, q\in \bfk$. Then the quadruple $(\mathrm{DT}(\Delta X),\prec,\succ, \bullet)$ equipped with the operator $d_X$ given in Eq.~(\mref{eq:defndiff}) is a differential $q$-tridendriform algebra of weight $\lambda$ on $X$.
\mlabel{pp:difftri}
\end{prop}

\begin{proof}
It follows from that the quadruple $(\mathrm{DT}(\Delta X),\prec,\succ, {\bullet})$
is a $q$-tridendriform algebra by Lemma~\mref{thm:freetridendriform} and
$d_X$ is a derivation of weight $\lambda$ on it by Eq.~(\mref{eq:defndiff}).
\end{proof}

We are ready for our main result in this section.
Let $j_X: X\rightarrow \mathrm{DT}(\Delta X),\, x\mapsto \stree x$.

\begin{theorem}
Let $X$ be a set and $\lambda,q \in \bfk$. Then the triple $(\mathrm{DT}(\Delta X), d_X, j_X)$ is the free differential $q$-tridendriform algebra of weight $\lambda$ on $X$.
\mlabel{thm:freedifftridendriform}
\end{theorem}

\begin{proof}
By Proposition~\ref{pp:difftri}, we are left to prove the universal property.
Let the quintuple $(T, \prec_T, \succ_T, \bullet_T, d_T)$  be a differential $q$-tridendriform algebra of weight $\lambda$ and let $f:X \rightarrow T$ be a set map.
To be compatible with the derivation, $f$ extends uniquely to a set map
\begin{equation*}
\tilde{f}:\Delta X\rightarrow T,\, x^{(n)}\mapsto d_T^n(f(x)).
\end{equation*}
By Lemma~\ref{thm:freetridendriform}, there is a unique $q$-tridendriform algebra homomorphism $\lbar{f}$ such that the following diagram is commutative
\[\xymatrix@C=2cm{
X\ar[r]\ar[dr]_{f} &\Delta X \ar@{.>}[d]_{\tilde{f}} \ar[r]^-{j_{\Delta X}}
                & \mathrm{DT}(\Delta X)
                \ar@{.>}[dl]_{\lbar{f}} \\
& T              }
\]
Here $j_{\Delta X}: \Delta X\rightarrow \mathrm{DT}(\Delta X),\, x^{(n)}\mapsto \stree {x^{(n)}}$.
It remains to show
\begin{equation}
\lbar{f}d_X(\ta) =d_T\lbar{f}(\ta),\quad \forall \ta\in \mathrm{DT}(\Delta X),
\mlabel{eq:dtf}
\end{equation}
which will be done by induction on $\dep(\ta)\geq 1$. For the initial step of $\dep({\ta})=1$, we reduce to
the induction on $\bre({\ta})\geq 2$. If $\bre({\ta})= 2$, we have ${\ta} = \Dtree{x^{(r_1)}}$ for some $x^{(r_1)}\in \Delta X$ and so
\begin{equation}
\begin{aligned}
\lbar{f}d_X({\ta}) =&\  \lbar{f}d_X\Big(\Dtree{x^{(r_1)}}\Big)
=\lbar{f}\Big(\Dtree{x^{(r_1+1)}}\Big) =\lbar{f} j_{\Delta X} (x^{(r_1+1)})=
 \tilde{f}(x^{(r_1+1)})\\
=& \ d^{r_1+1}_T(f(x)) = d_T(d^{r_1}_T(f(x)))
=d_T(\tilde{f}(x^{(r_1)})\\
=& d_T\bar{f}  j_{\Delta X}(x^{(r_1)})
=d_T\lbar{f}\Big(\Dtree{x^{(r_1)}}\Big) = d_T \lbar{f}({\ta}).
\mlabel{eq:initial}
\end{aligned}
\end{equation}
For a given $m\geq 2$, assume that Eq.~\eqref{eq:dtf} holds for $\bre({\ta})=m$, and consider the case of $\bre({\ta})=m+1$. Denote by
\begin{equation*}
{\ta}=\DDDtree{x_1^{(r_1)}}{x_m^{(r_m)}} = \Dtree{x_1^{(r_1)}}  \bullet  \DDDtree{x_2^{(r_2)}}{x_m^{(r_m)}}=: \Dtree{x^{(r_1)}_1} \bullet  {\ta}'.
\end{equation*}
Then
\begin{align*}
\lbar{f}d_X({\ta})=& \ \lbar{f}d_X\Big(\Dtree{x^{(r_1)}_1} \bullet  {\ta}'\Big)\\
=& \ \lbar{f}\Big(d_X\Big(\Dtree{x^{(r_1)}_1}\Big) \bullet {\ta}'
+\Dtree{x^{(r_1)}_1} \bullet d_X({\ta}')
+\lambda d_X\Big(\Dtree{x^{(r_1)}_1}\Big) \bullet d_X({\ta}')\Big)\quad\text{(by Eq.~\eqref{eq:difftri3})}\\
=& \Big(\lbar{f}d_X\Big(\Dtree{x^{(r_1)}_1}\Big)\bullet_T\lbar{f}({\ta}')+\lbar{f}
\Big(\Dtree{x^{(r_1)}_1}\Big)\bullet_T\lbar{f}
d_X({\ta}')
+\lambda\lbar{f}d_X\Big(\Dtree{x^{(r_1)}_1}\Big)\bullet_T\lbar{f}d_X({\ta}')\Big),\\
&\hspace{3cm}\text{(by $\lbar{f}$ being a $q$-tridendriform algebra homomorphism)}\\
=& \ d_T\lbar{f}\Big(\Dtree{x^{(r_1)}_1}\Big)\bullet_T\lbar{f}({\ta}')+ \lbar{f}\Big(\Dtree{x^{(r_1)}_1}\Big)\bullet_T d_T\lbar{f}({\ta}')
+\lambda d_T\lbar{f}\Big(\Dtree{x^{(r_1)}_1}\Big)\bullet_T d_T\lbar{f}({\ta}'),\\
&\hspace{3cm}\text{(by Eq.~\eqref{eq:initial} and the induction hypothesis on breadth)}\\
=& \ d_T\Big(\lbar{f}\Big(\Dtree{x^{(r_1)}_1}\Big)\bullet_T\lbar{f}({\ta}')\Big)\quad\text{(by Eq.~\eqref{eq:difftri3})}\\
=& \ d_T\lbar{f}\Big(\Dtree{x^{(r_1)}_1} \bullet  {\ta}'\Big)\quad\text{(by $\lbar{f}$ being a $q$-tridendriform algebra homomorphism)}\\
=& \ d_T\lbar{f}({\ta}).
\end{align*}

For the induction step $\dep({\ta})\geq 2$, we again reduce to the induction on
$\bre({\ta})\geq 2$. If $\bre({\ta})=2$, then
{\small{\begin{align*}
\lbar{f}d_X({\ta})=& \ \lbar{f}d_X\Big(\bigvv {x_1^{(r_1)}};2;({\ta}^{(0)},{\ta}^{(1)})\Big)
=\lbar{f}d_X\Big(\Big({\ta}^{(0)}\succ\Dtree{x^{(r_1)}_1}\Big)\prec {\ta}^{(1)}\Big)\\
=& \ \lbar{f}\Bigg(d_X\Big({\ta}^{(0)}\succ\Dtree{x^{(r_1)}_1}\Big)\prec {\ta}^{(1)}
+\Big({\ta}^{(0)}\succ\Dtree{x^{(r_1)}_1}\Big)\prec d_X({\ta}^{(1)})
+\lambda d_X\Big({\ta}^{(0)}\succ\Dtree{x^{(r_1)}_1}\Big)\prec d_X({\ta}^{(1)})\Bigg)\\
&\hspace{7cm}\text{(by Eq.~\eqref{eq:difftri2})}\\
=& \ \lbar{f}d_X\Big({\ta}^{(0)}\succ\Dtree{x^{(r_1)}_1}\Big)\prec \lbar{f}({\ta}^{(1)})
+\lbar{f}\Big({\ta}^{(0)}\succ\Dtree{x^{(r_1)}_1}\Big)\prec_T \lbar{f}d_X({\ta}^{(1)})\\
& \ +\lambda \lbar{f}d_X\Big({\ta}^{(0)}\succ\Dtree{x^{(r_1)}_1}\Big)\prec_T\lbar{f}d_X({\ta}^{(1)}) \hspace{0.5cm}\text{(by $\bar{f}$ being a $q$-tridendriform algebra homomorphism)}\\
=& \ \lbar{f}\Bigg(d_X({\ta}^{(0)})\succ \Dtree{x^{(r_1)}_1}+{\ta}^{(0)}\succ d_X\Big(\Dtree{x^{(r_1)}_1}\Big)+\lambda d_X({\ta}^{(0)})\succ d_X\Big(\Dtree{x^{(r_1)}_1}\Big)\Bigg)\prec_T\lbar{f}({\ta}^{(1)})\\
&+\Big(\lbar{f}({\ta}^{(0)})\succ_T\lbar{f}\Big(\Dtree{x^{(r_1)}_1}\Big)\Big)\prec_T \lbar{f}d_X({\ta}^{(1)})\\
&+\lambda  \lbar{f}\Bigg(d_X({\ta}^{(0)})\succ \Dtree{x^{(r_1)}_1}+{\ta}^{(0)}\succ d_X\Big(\Dtree{x^{(r_1)}_1}\Big)+\lambda d_X({\ta}^{(0)})\succ d_X\Big(\Dtree{x^{(r_1)}_1}\Big)\Bigg)\prec_T\lbar{f}d_X({\ta}^{(1)})\\
&\hspace{7cm}\text{(by Eq.~\eqref{eq:difftri4})}\\
=& \ \Bigg(\lbar{f}d_X({\ta}^{(0)})\succ_T\lbar{f}\Big(\Dtree{x^{(r_1)}_1}\Big)+\lbar{f}({\ta}^{(0)})\succ_T
\lbar{f}d_X\Big(\Dtree{x^{(r_1)}_1}\Big)+\lambda\lbar{f}d_X({\ta}^{(0)})\succ_T
\lbar{f}d_X\Big(\Dtree{x^{(r_1)}_1}\Big)\Bigg)\prec_T\lbar{f}({\ta}^{(1)})\\
&+\Big(\lbar{f}({\ta}^{(0)})\succ_T\lbar{f}\Big(\Dtree{x^{(r_1)}_1}\Big)\Big)\prec_T\lbar{f}d_X({\ta}^{(1)})\\
&+\lambda\Bigg(\lbar{f}d_X({\ta}^{(0)})\succ_T\lbar{f}\Big(\Dtree{x^{(r_1)}_1}\Big)
+\lbar{f}({\ta}^{(0)})\succ_T\lbar{f}d_X\Big(\Dtree{x^{(r_1)}_1}\Big)
+\lambda\lbar{f}d_X({\ta}^{(0)})\succ_T\lbar{f}d_X\Big(\Dtree{x^{(r_1)}_1}\Big)\Bigg)\prec_T\lbar{f}d_X({\ta}^{(1)})\\
&\hspace{3cm}\text{(by $\bar{f}$ being a $q$-tridendriform algebra homomorphism)}\\
=& \ \Bigg(d_T\lbar{f}({\ta}^{(0)})\succ_T\lbar{f}\Big(\Dtree{x^{(r_1)}_1}\Big)+\lbar{f}({\ta}^{(0)})\succ_T
d_T\lbar{f}\Big(\Dtree{x^{(r_1)}_1}\Big)+\lambda d_T\lbar{f}({\ta}^{(0)})\succ_T
d_T\lbar{f}\Big(\Dtree{x^{(r_1)}_1}\Big)\Bigg)\prec_T\lbar{f}({\ta}^{(1)})\\
&+\Big(\lbar{f}({\ta}^{(0)})\succ_T\lbar{f}\Big(\Dtree{x^{(r_1)}_1}\Big)\Big)\prec_T d_T\lbar{f}({\ta}^{(1)})\\
&+\lambda\Bigg(d_T\lbar{f}({\ta}^{(0)})\succ_T\lbar{f}\Big(\Dtree{x^{(r_1)}_1}\Big)
+\lbar{f}({\ta}^{(0)})\succ_T d_T\lbar{f}\Big(\Dtree{x^{(r_1)}_1}\Big)
+\lambda d_T\lbar{f}({\ta}^{(0)})\succ_T d_T\lbar{f}\Big(\Dtree{x^{(r_1)}_1}\Big)\Bigg)\prec_T d_T\lbar{f}({\ta}^{(1)})\\
&\hspace{3cm}\text{(by the induction hypothesis on depth and Eq.~\eqref{eq:initial})}\\
=& \ d_T\Bigg(\lbar{f}({\ta}^{(0)})\succ_T\lbar{f}\Big(\Dtree{x^{(r_1)}_1}\Big)\Bigg)\prec_T\lbar{f}({\ta}^{(1)})
+\Big(\lbar{f}({\ta}^{(0)})\succ_T\lbar{f}\Big(\Dtree{x^{(r_1)}_1}\Big)\Big)\prec_T d_T\lbar{f}({\ta}^{(1)})\\
&+\lambda  d_T\Bigg(\lbar{f}({\ta}^{(0)})\succ_T\lbar{f}\Big(\Dtree{x^{(r_1)}_1}\Big)\Bigg)\prec_T d_T\lbar{f}({\ta}^{(1)}) \hspace{2cm} \text{(by Eq.~\eqref{eq:difftri4})}\\
=& \ d_T\Bigg(\Big(\lbar{f}({\ta}^{(0)})\succ_T\lbar{f}\Big(\Dtree{x^{(r_1)}_1}\Big)\Big)\prec_T\lbar{f}({\ta}^{(1)})\Bigg)
\hspace{2cm} \text{(by Eq.~\eqref{eq:difftri2})}\\
=&\ d_T\lbar{f}\Big(\Big({\ta}^{(0)}\succ\Dtree{x^{(r_1)}_1}\Big)\prec {\ta}^{(1)}\Big)
\hspace{1cm}\text{(by $\bar{f}$ being a $q$-tridendriform algebra homomorphism)}\\
=&\ d_T\lbar{f}({\ta}).
\end{align*}}}
For a fixed $m\geq 2$, assume that Eq.~\eqref{eq:dtf} is valid for $\bre({\ta})=m$,
and consider the case of $\bre({\ta})=m+1$.
Denote by \begin{equation*}
{\ta}^{0,1}:=\bigvv x^{(r_1)}_1;2;({\ta}^{(0)},{\ta}^{(1)})\,\text{ and }\, \lbar{{\ta}} :=\bigvv x^{(r_2)}_2,\ldots,x^{(r_m)}_m;m;(|, {\ta}^{(2)},\cdots,{\ta}^{(m)}).
\end{equation*}
Then ${\ta} = {\ta}^{0,1} \bullet  \lbar{{\ta}}$ and
\begin{align*}
\lbar{f}d_X({\ta})=& \ \lbar{f}d_X({\ta}^{0,1} \bullet  \lbar{{\ta}})\\
=& \ \lbar{f}\Big(d_X({\ta}^{0,1}) \bullet  \lbar{{\ta}}+{\ta}^{0,1} \bullet  d_X( \lbar{{\ta}})+\lambda d_X({\ta}^{0,1}) \bullet  d_X( \lbar{{\ta}})\Big)\quad\text{(by Eq.~\eqref{eq:difftri3})}\\
=& \ \lbar{f}d_X({\ta}^{0,1})\bullet_T\lbar{f}( \lbar{{\ta}})+\lbar{f}({\ta}^{0,1})\bullet_T\lbar{f}d_X( \lbar{{\ta}})
+\lambda\lbar{f}d_X({\ta}^{0,1})\bullet_T\lbar{f}d_X( \lbar{{\ta}})\\
&\hspace{3cm}\text{(by $\lbar{f}$ being a $q$-tridendriform algebra homomorphism)}\\
=& \ d_T\lbar{f}({\ta}^{0,1})\bullet_T\lbar{f}( \lbar{{\ta}})+\lbar{f}({\ta}^{0,1})\bullet_T d_T\lbar{f}( \lbar{{\ta}})
+\lambda d_T\lbar{f}({\ta}^{0,1})\bullet_T d_T\lbar{f}( \lbar{{\ta}})\\
&\hspace{3cm}\text{(by the induction hypothesis on breadth)}\\
=& \ d_T\Big(\lbar{f}({\ta}^{0,1})\bullet_T\lbar{f}( \lbar{{\ta}})\Big)\hspace{1cm}\text{(by Eq.~\eqref{eq:difftri3})}\\
=&\ d_T\lbar{f}({\ta}^{0,1} \bullet   \lbar{{\ta}})\hspace{1cm}\text{(by $\lbar{f}$ being a $q$-tridendriform algebra homomorphism)}\\
=& \ d_T\lbar{f}({\ta}).
\end{align*}
This completes the proof.
\end{proof}

\subsection{Free weighted differential dendriform algebras}
\mlabel{sec:freediffdend}
In this subsection, we construct free differential dendriform algebras of weight $\lambda$ via planar binary trees.

\subsubsection{Free dendriform algebras}
\mlabel{sub:fdd}
Let $X$ be a set. For $n\geq 0$, let $Y_{n,\,X}$ be the set of planar binary trees with $n+1$ leaves and with internal vertices decorated by elements of $X$. The unique tree with one leaf is denoted by $|$.
Here are the first few of them.
\begin{align*}
Y_{0,\,X}&=\{|\},\ \
Y_{1,\,X}=\left\{ \stree x\Bigm| x\in X \right\},\ \
Y_{2,\,X}=\left\{
\XX{\xxr{-5}5
\xxhu00x \xxhu{-5}5y
}, \,
\XX{\xxl55
\xxhu00x \xxhu55y
}\Bigm| x,y\in X
\right\},\\
Y_{3,\,X}&=\left\{
\XX[scale=1.6]{\xxr{-4}4\xxr{-7.5}{7.5}
\xxhu00{x} \xxhu[0.1]{-4}4{y} \xxhu[0.1]{-7.5}{7.5}{z}
}, \,
\XX[scale=1.6]{\xxl44\xxl{7.5}{7.5}
\xxhu00{x} \xxhu[0.1]44{y} \xxhu[0.1]{7.5}{7.5}{z}
}, \,
\XX[scale=1.6]{\xxr{-6}6\xxl66
\xxhu00{x} \xxhu[0.1]66{y} \xxhu{-6}6{z}
}, \,
\XX[scale=1.6]{\xxr{-5}5\xxl{-2}8
\xxhu00x
\xxhu[0.1]{-5}5{y} \xxhu[0.1]{-2}8{z}
}, \,
\XX[scale=1.6]{\xxl55\xxr28
\xxhu00x
\xxhu[0.1]55{\,y} \xxhu[0.1]28{z}
}\, \Biggm| x,y,z\in X\right\}.
\end{align*}

For $\ta\in Y_{m,\,X}, \sigm\in Y_{n,\,X}$ and $x\in X$, the grafting $\vee_x$ of $\ta$ and $\sigm$ over the vertex $x$ is defined to be the planar binary tree $\ta\vee_x \sigm\in Y_{m+n+1,\,X}$ obtained by adding a new vertex decorated by $x$ and joining the roots of $\ta$ and $\sigm$ to the new vertex.

Given a planar binary tree $\ta\in Y_{n,\,X}$ not equal to $|$, there is a unique decomposition $\ta=\ta^{l}\vee_x \ta^{r}$ for some $x\in X$. For example,
$$ \stree x=|\vee_x|,\, \XX{\xxr{-5}5
\xxhu00x \xxhu{-5}5y
}=\stree y\vee_x|,\, \XX{\xxl55
\xxhu00x \xxhu55y
}=|\vee_x\stree y.$$

Let $\mathrm{DD}(X):=\underset{n\geq 1}\bigoplus\,\bfk Y_{n,\,X}$. Define binary operations $\prec$ and $\succ$ on $\mathrm{DD}(X)$
recursively as follows.
\begin{enumerate}
\item For ${\ta}\in Y_{n,\,X}$ with $n\geq 1$, define
$$|\succ {\ta}:={\ta}\prec |:={\ta}\,\text{ and }\, |\prec {\ta}:={\ta}\succ |:=0.$$

\item For ${\ta}={\ta}^{l}\vee_{x} {\ta}^{r}$ and $\sigm=\sigm^{l}\vee_{y} \sigm^{r},$ put
\begin{equation*}
{\ta}\prec \sigm:={\ta}^{l}\vee_{x} ({\ta}^{r}\prec \sigm+{\ta}^{r}\succ \sigm)\,\text{ and }\, {\ta}\succ \sigm:=({\ta}\prec \sigm^{l}+{\ta}\succ \sigm^{l})\vee_{y} \sigm^{r}.
\end{equation*}
\end{enumerate}

The following result is the construction of the free dendriform algebra.

\begin{lemma}\cite{LoRo98}
Let $X$ be a set. The triple $(\mathrm{DD}(X),\prec,\succ),$ together with the natural embedding
\[
j_X: X \rightarrow \mathrm{DD}(X), \, x\mapsto \stree x,
\]
is the free dendriform algebra on $X.$
\mlabel{thm:freedendriform}
\end{lemma}

\subsubsection{Free weighted differential dendriform algebras}
\label{sub:fdda}
In this subsection, we give a differential version of Lemma~\mref{thm:freedendriform}.

Let $X$ be a set. Denote by
$\Delta X :=X \times \mathbb{N}=\left\{x^{(n)} \mid x \in X, n \geq 0\right\}$ as before. Define two binary operations $\prec$ and $\succ$ on $DD(\Delta X)$ recursively on $\dep(\ta) + \dep(\sigm)$ as follows.
\begin{enumerate}
\item For ${\ta}\in Y_{n,\,\Delta X}$ with $n\geq 1$,
set $$|\succ {\ta}:={\ta}\prec |:={\ta}\,\text{ and }\, |\prec {\ta}:={\ta}\succ |:=0.$$

\item For ${\ta}={\ta}^{l}\vee_{x^{(n)}} {\ta}^{r}$ and $\sigm=\sigm^{l}\vee_{y^{(m)}} \sigm^{r}$, define
\begin{equation*}
{\ta}\prec \sigm:={\ta}^{l}\vee_{x^{(n)}} ({\ta}^{r}\prec \sigm+{\ta}^{r}\succ \sigm),\quad {\ta}\succ \sigm:=({\ta}\prec \sigm^{l}+{\ta}\succ \sigm^{l})\vee_{y^{(m)}} \sigm^{r}.
\end{equation*}
\end{enumerate}

Since $\mathrm{DD}(\Delta X)$ is the free dendriform algebra on $\Delta X$ by Lemma~\mref{thm:freedendriform},
each $\ta\in \mathrm{DD}(\Delta X)$ can be uniquely written as
\begin{equation*}
  \ta=f_\ta(x_1^{(r_1)},\cdots, x_n^{(r_n)})
\end{equation*}
for some $x_1^{(r_1)},\cdots, x_n^{(r_n)}\in\Delta X$ and some $n$-arity operation $f_\ta$ in the operad of dendriform algebras.
Applying Proposition~\mref{prop:wdiff}, the map
$$\Delta X \rightarrow \mathrm{DD}(\Delta X), \, x^{(n)}\mapsto x^{(n+1)}$$ extends to a unique derivation
$d_X: \mathrm{DD}(\Delta X) \rightarrow \mathrm{DD}(\Delta X)$ with
\begin{eqnarray}\label{eq:defndiffdend}
d_X({\ta})\nonumber&=&d_X\big(f_{\ta}(x_1^{(r_1)},\cdots, x_n^{(r_n)})\big)\nonumber\\
&=&
\sum_{k=1}^{n} \lambda^{k-1}\bigg(\sum_{1\leq i_1<\cdots< i_k \leq n}
f_{\ta}(x_1^{(r_1)},\cdots, x_{i_1}^{(r_{i_1}+1)},\cdots,x_{i_2}^{(r_{i_2}+1)},\cdots,x_{i_k}^{(r_{i_k}+1)},
\cdots,x_n^{(r_n)})\bigg).\nonumber
\end{eqnarray}

Now we are ready for the main result in this subsection.

\begin{theorem}
Let $X$ be a set and $\lambda \in \bfk$. The quadruple $(\mathrm{DD}(\Delta X),\prec,\succ, d_X)$, together with the natural embedding
$$ j_X: X\rightarrow \mathrm{DD}(\Delta X),\, x\mapsto \stree x,$$
is the free differential dendriform algebra of weight $\lambda$ on $X$.
\mlabel{thm:freediffdendriform}
\end{theorem}

\begin{proof}
The proof is similar to the one of Theorem~\ref{thm:freedifftridendriform}.
\end{proof}

\bigskip

\noindent
{{\bf Acknowledgments.} This work is supported by Natural Science Foundation of China (12071191, 12101183).
X. Gao is also supported by Innovative Fundamental Research Group Project of Gansu
Province (23JRRA684).  Y. Y. Zhang is also supported by China Postdoctoral Science Foundation (2021M690049).

\medskip

\noindent
{\bf Competing Interests.} There is no conflict of interest.

\medskip

\noindent
{\bf Data Availability.} The manuscript has no associated data.

\end{document}